\documentclass[a4paper]{amsart}
\usepackage{mathrsfs}
\usepackage{amssymb}
\usepackage{amscd}
\usepackage[all]{xy}
\usepackage{amsmath}
\usepackage{dsfont}
\usepackage{stmaryrd}
\usepackage{currvita}
\usepackage{hyperref}
\usepackage[latin1]{inputenc}
\usepackage{ae}

\usepackage{amsthm}

\usepackage{xypic}

\usepackage{ifthen}
\xyoption{curve}

\newtheorem{lemma}{Lemma}[section]
\newtheorem{prop}[lemma]{Proposition}

\newtheorem{thm}[lemma]{Theorem}
\newtheorem{cor}[lemma]{Corollary}
\newtheorem{conj}[lemma]{Conjecture}

\theoremstyle{definition}
\newtheorem{defn}[lemma]{Definition}
\newtheorem{remark}[lemma]{Remark}
\newtheorem{example}[lemma]{Example}

\def\phi{\varphi}
\def\theta{\vartheta}
\def\epsilon{\varepsilon}

\let\tensor\varotimes
\let\oplus\varoplus

\def\iso{\cong}

\newcommand\defeq{:=}

\newcommand\set[2][auto]{
     \ifthenelse{\equal{#1}{auto}}{\left\lbrace}{\csname #1\endcsname\lbrace} #2 
     \ifthenelse{\equal{#1}{auto}}{\right\rbrace}{\csname #1\endcsname\rbrace} }

\def\mathbb{\mathds}

\newcommand\bbN{{\mathbb{N}}}
\newcommand\bbZ{{\mathbb{Z}}}
\newcommand\bbQ{{\mathbb{Q}}}
\newcommand\bbR{{\mathbb{R}}}
\newcommand\bbC{{\mathbb{C}}}

\newcommand\bbF{{\mathbb{F}}}

\newcommand\bbP{{\mathbb{P}}}

\newcommand\scrO{\mathscr O}

\newcommand\calB{\mathcal B}

\newcommand\calE{\mathcal E}

\newcommand\frakX{{\mathfrak{X}}}

\newcommand\frakg{{\mathfrak{g}}}

\newcommand\frako{\mathfrak o}

\newcommand\longto{\longrightarrow}

\newcommand\isoto{\mbox{$\hspace{7.5pt}\raise 3pt\hbox{$\sim$}\hspace{-17pt}\longrightarrow\hspace{3pt}$}\linebreak[0]}

\newcommand\isoot{\mbox{$\hspace{8.5pt}\raise 3pt\hbox{$\sim$}\hspace{-18pt}\longleftarrow\hspace{3pt}$}\linebreak[0]}

\newcommand\into{\hookrightarrow}

\newcommand\longinto{\lhook \joinrel \longrightarrow}

\newcommand\onto{\to\nolinebreak\hspace{-9.5pt}\to}

\newcommand\longonto{\longto\hspace{-9.5pt}\to}

\newcommand\To[1]{\stackrel {#1}\to}

\newcommand\Longto[1]{\stackrel {#1}\longto}

\DeclareMathOperator{\preim}{pre\kern0.13em im} %
\DeclareMathOperator{\preker}{pre\kern0.13em ker} %
\DeclareMathOperator{\precoker}{pre\kern0.13em coker} %
\renewcommand{\projlim}{\operatorname*{\underleftarrow{\lim}}}

\begin{document}
\title{Monodromy groups of vector bundles on p-adic curves}
\author{Ralf Kasprowitz} 

\address{Universit\"at Paderborn, Fakult\"at f\"ur Elektrotechnik, Informatik und Mathematik,
Institut f\"ur Mathematik, Warburger Str. 100, 33098 Paderborn, Germany}
%\curraddr{}

\email{kasprowi@math.upb.de}

\begin{abstract}In \cite{dw2} and \cite{dw3}, C. Deninger and A. Werner developed a partial $p$-adic analogue of the classical Narasimhan-Seshadri correspondence between vector bundles and representations of the fundamental group. We will investigate the various monodromy groups that occur in this theory, that is the image of these representations and their Zariski closure as well as the Tannaka dual group of these vector bundles.    \end{abstract}

\maketitle

\section{Introduction}

There are several approaches to extend the classical correspondences between vector bundles and representations of the fundamental group to the $p$-adic case. There is the work of C. Deninger and A. Werner, mainly in \cite{dw2}, that partially extends the Narasimhan-Seshadri correspondence for vector bundles on Riemannian surfaces to $p$-adic curves. In \cite{fal}, G. Faltings even developes a $p$-adic analogue of the Simpson correspondence for Higgs bundles.
\medskip

In this paper, we will use the explicit construction of Deninger and Werner to investigate the monodromy groups of the bundles to which one can attach a representation of the \'etale fundamental group. More precicely, we are concerned with the structure of the image of this representation, its Zariski closure and the Tannaka dual group associated to such vector bundles. The image of the associated representation is called the analytic monodromy group, while its Zariski closure and the Tannaka dual group are the algebraic monodromy groups.\medskip

We now describe the main results of this paper. Let $X$ be a smooth, connected and projective curve over $\overline{\bbQ}_p$ and $E$ a vector bundle on $X_{\bbC_p}$. If there is a model $\frakX$ of the curve $X$ and a model $\calE$ on $\frakX$, such that the reduction modulo $p^q$ of the bundle $\calE$ is trivial for some $q \in \bbQ^+$, we call the vector bundle $E$ trivial modulo $p^q$. If $E$ is semistable of degree 0, we denote by $G_E$ the Tannaka dual group of the Tannaka subcategory generated by $E$. We prove the following Theorem:

\begin{thm} \label{result1}
 Let $E$ be a vector bundle of rank $r$ on the smooth, connected and projective curve $X_{\bbC_p}$. If $E$ is trivial modulo $p^q$, with $q > \frac{1}{p-1}$ for odd primes $p$ and $q \geq 1$ for $p=2$, then the algebraic monodromy group $G_E$ is connected.
\end{thm}

We will apply this Theorem to the restrictions of certain stable vector bundles on the projective space $\bbP^r_{\bbC_p}$ and obtain:

\begin{thm} \label{result2}
 Let $E$ be a stable vector bundle of degree $0$ and rank $r$ on the projective space $\bbP^r_{\bbC_p}$ having a resolution $$0 \longto \bigoplus_{i = 1}^c \scrO(a_i) \longto \bigoplus_{j = 1}^{c+r}\scrO(b_j) \longto E \to 0.$$ Let $X \subset \bbP^r_{\overline{\bbQ}_p}$ be a smooth connected curve of degree $d>>0$. If the restriction of $E$ to $X_{\bbC_p}$ is trivial modulo $p^q$ for some $q$ as in Theorem \ref{result1}, one has:

\begin{enumerate}
\item The vector bundle $E$ on the smooth projective curve $X_{\bbC_p}$ is stable and lies in the category $\mathfrak{B}^s_{X_{\bbC_p}}$.
\item The Tannaka dual group $G_E \subset \mathbf{GL}_{E_x}$ of $E$ restricted to the curve $X_{\bbC_p}$ is connected and semisimple. All components of $G_E$ are of type $A$.

\item If $\dim(E_x^{\tensor r})^{G_E} = \dim\Gamma(X_{\bbC_p}, E^{\tensor r}) \leq 3$ or if $r$ is some prime-power, then $G_E$  is almost simple of type $A$. 
\end{enumerate}
\end{thm}

The last result concerns the analytic monodromy group. It is a natural question under which conditions the image $G_{\rho_E}$ of the associated representation of the fundamental group $\rho_E: \pi_1(X,x) \to \mbox{GL}(E_x)$ is algebraic, i.e. such that there exists a change of basis of the $\bbC_p$-vector space $E_x$ where $G_{\rho_E}$ lies in $\mbox{GL}_r(\overline{\bbQ}_p)$.

\begin{thm} \label{result3}
 Let $E$ be a polystable vector bundle with potentially strongly semi\-stable reduction of degree $0$ on a smooth, connected and projective curve $X_{\bbC_p}$, such that some power of the determinant bundle is trivial. Furthermore, let $\rho_E$ be a semisimple representation. Then $G_{\rho_E}$ is virtually algebraic if and only if $G_{\rho_E}$ is $p$-adic analytic.
\end{thm}

Here virtually algebraic means that the group $G_{\rho_E}$ has a subgroup of finite index which already lies in $\mbox{GL}(\overline{\bbQ}_p)$, which is equivalent to the existence of a finite \'etale covering $f:Y \to X$ such that $G_{\rho_{f^*E}}$ is defined over $\overline{\bbQ}_p$. \bigskip

Let us now describe the content of this paper in more detail. In the second section we introduce the theory of Deninger and Werner. The most important results are explained as well as some more technical details of the construction, as far as these are necessary for the proofs of this paper.\medskip

The third section deals with algebraic monodromy groups of vector bundles on p-adic curves. It was already noted in \cite{dw3} that the Tannaka dual group of a vector bundle $E$ with potentially strongly semistable reduction of degree $0$ coincides with the Zariski closure of the image of the associated representation of the \'etale fundamental group if the representation is semisimple. We will give a proof of this fact here. But the main result of this section is Theorem \ref{result1} described above. This can be used for explicit calculations of the Tannaka dual group by computing global sections of $n$-fold tensor products of $E$, which will be done for certain kernel bundles in the fourth section. The third section ends with a nice criterion due to M. Larsen, which also allows to compute the Tannaka dual group of polystable vector bundles in several cases, see Proposition \ref{larsen}.\medskip

In the fourth section we consider stable kernel bundles of degree $0$ and rank $r$ on the projective space $\bbP^r$, that is vector bundles $E$ sitting in a short exact sequence $0 \to E \to \bigoplus_{i = 1}^c \scrO(a_i) \to \bigoplus_{j = 1}^{c+r}\scrO(b_j) \to 0$. We will show that for smooth curves of sufficiently high degree in $\bbP^r$ such that the Tannaka dual group of the restriction of the vector bundle $E$ is connected, this group has to be semisimple and all components are of type $A$. This is Theorem \ref{result2}. We conjecture that the group is almost simple, namely that one always obtains the standard $r$-dimensional representation of $\mbox{SL}_r$ or its dual.  Furthermore, we explain how to produce many examples of such bundles that satisfy the conditions of this theorem for a generic choice of a smooth connected curve of sufficiently high degree using the main result of the third section.\medskip

The fifth section is devoted to the central part of the proof of Theorem \ref{result2} using the representation theory of semisimple Lie algebras. We prove the following, see Proposition \ref{invar}: Let $\frakg$ be a semisimple Lie algebra over an algebraically closed field of characteristic $0$ and $V$ an $r$-dimensional, irreducible and faithful $\frakg$-module that satisfies $$(V^{\tensor n})^{\frakg} = \{0\} \mbox{ for }1 \leq n < r.$$ Then all components of $\frakg$ are of type $A$.\medskip

The last section deals with properties of the analytic monodromy group of a vector bundle with strongly semistable reduction of degree $0$. We will prove Theorem \ref{result3} above and give a criterion for a vector bundle to have a $p$-adic analytic monodromy group using the group-theoretical classification  of $p$-adic analytic groups by Lazard, see Corollary \ref{Laza}. It only involves counting the degree of certain trivializing coverings of the curve $X$, but it seems to be very hard to apply for a concretely given vector bundle. Finally, we exploit the compatibility of the Galois-actions on the category of vector bundles with potentially strongly semistable reduction and the category of representations of the fundamental group to give some further information about the analytic monodromy group. In several cases, this action can be described using results of J.-P. Serre and S. Sen on semilinear Galois representations of Hodge-Tate type. \medskip    

This paper is an outgrowth of my thesis. It is a pleasure to thank Christopher Deninger for his advice during that time. I am also grateful to Annette Werner and Torsten Wedhorn for valuable comments.     

\section{Vector bundles on $p$-adic curves}

In this section we will briefly sketch the main results of Deninger and Werner and give some technical details of their construction as far as this is important for the proofs in this paper. The main reference is \cite{dw2}. The following notations are used throughout the paper. Let $p$ be a prime number and let $\bbQ_p$ be the field of $p$-adic numbers. We denote by $\bbC_p := \widehat{\overline{\bbQ}}_p$ the completion of an algebraic closure of $\bbQ_p$ and by $\frako \subseteq \bbC_p$ its valuation ring with maximal ideal $\mathfrak{m}$ and residue field $k := \frako/\mathfrak{m} = \overline{\bbF}_p$. By a curve over a field $K$ we always mean a purely one-dimensional separated scheme of finite type over $K$. If $X$ is a smooth, connected and projective curve over $\overline{\bbQ}_p$, we call a finitely presented, proper and flat scheme $\frakX$ over $\overline{\bbZ}_p$ with generic fibre isomorphic to $X$ a model of the curve $X$. We denote by $X_{\bbC_p}, \frakX_{\frako}$ resp. $\frakX_k$ the base change $X \tensor_{\overline{\bbQ}_p}\bbC_p ,\frakX \tensor_{\overline{\bbZ}_p} \frako$ resp. $\frakX \tensor_{\overline{\bbZ}_p} k$.\medskip

By semistability of vector bundles of rank $r$ on a smooth, projective varieties of dimension $n$ over $\bbC_p$ we always mean slope semistability, i.e. we fix an ample line bundle $\scrO(1)$ on $X$ together with a corresponding divisor $H$ and define $\mu(E) := c_1(E)\cdot H^{n-1}/r$. We call $E$ slope $H$-semistable if for all proper torsion-free subsheaves $F \subset E$ the inequality $\mu(F) \leq \mu(E)$ holds (and stable if it is strictly smaller). For curves and the projective space $\bbP^n$, the slope $\mu$ is just the degree of the bundle divided by its rank.\medskip

 Finally, for a curve $C$ over $k$ let $F: C \to C$ be the absolute Frobenius. It is defined by the identity on the underlying topological space and the $p$-power map on the structure sheaf. It sits in the commutative diagram $$\xymatrix{ C \ar[r]^F \ar[d] & C \ar[d] \\ \mbox{spec}\,k \ar[r]^F & \mbox{spec}\,k.}$$ The lower horizontal morphism is induced by the Frobenius morphism $F: k \to k$. 

\begin{defn}
Let $C$ be a curve over $k$.
\begin{enumerate} \item If $C$ is smooth, connected and projective, we call a vector bundle $E$ on $C$ strongly semistable of degree $0$ if $\mbox{deg}(E) = 0$ holds and the pullback $F^{n*}E$ by the $n$-times iterated Frobenius morphism is semistable for all $n \geq 1$.
\item If $C$ is proper, with irreducible components $C_i$ for $i = 1,\dots,n$, and $\tilde{C}_i$ their normalizations with canonical morphisms $\alpha_i: \tilde{C}_i \to C_i \to C$, we call a vector bundle $E$ on $C$ strongly semistable of degree $0$, if all $\alpha_i^*(E)$ are strongly semistable of degree $0$ on the smooth, connected and projective curves  $\tilde{C}_i$. \end{enumerate} \end{defn}
We will now describe certain subcategories of vector bundles on the curve $X_{\bbC_p}$. 

\begin{defn} Let $E$ be a vector bundle on the smooth, connected and projective curve $X_{\bbC_p}$.\begin{enumerate} \item The vector bundle $E$ has strongly semistable reduction of degree $0$, if there exists a model $\frakX$ of $X$ and a vector bundle  $\mathcal{E}$ on $\frakX$, such that $E$ is isomorphic to the generic fibre of $\mathcal{E}$ and the special fibre  $\mathcal{E}_k$ is strongly semistable of degree $0$ on the proper $k$-curve $\frakX_k$. We denote the full subcategory of such bundles with $\mathfrak{B}^s_{X_{\bbC_p}}$. \item The vector bundle $E$ has potentially strongly semistable reduction of degree $0$, if there is a finite morphismus $\alpha: Y \to X$ of smooth, projective curves over $\overline{\bbQ}_p$, such that the pullback $\alpha^*_{\bbC_p}(E)$ has strongly semistable reduction of degree $0$ on the curve $Y_{\bbC_p}$. We denote by $\mathfrak{B}^{ps}_{X_{\bbC_p}}$ the full subcategory of these vector bundles. 
\end{enumerate} \end{defn}  

Furthermore, we will need certain subcategories of vector bundles on the model $\frakX_{\frako}$.

\begin{defn}
Let $R$ be a valuation ring with quotient field $Q$ and $\frakX$ a model of the smooth, connected and projective curve $X$ over $Q$. \begin{enumerate} \item The category $\mathcal{S}_{\frakX, D}$, where $D$ is a divisor on $X$, consists of proper, finitely presented $R$-morphisms $\pi: \mathcal{Y} \to \frakX$, such that the generic fibre $\pi_{Q}: \mathcal{Y}_{Q} \to X$ is finite and $\pi^{-1}_Q(X\backslash D) \to X\backslash D$ an \'etale morphism.  A morphism from $\pi_1: \mathcal{Y}_1 \to \frakX$ to $\pi_2: \mathcal{Y}_2 \to \frakX$ is a morphism $\phi: \mathcal{Y}_1 \to \mathcal{Y}_2$ with $\pi_1 = \pi_2 \circ \phi$. 
\item The full subcategory $\mathcal{S}^{good}_{\frakX, D} \subset \mathcal{S}_{\frakX,D}$ consists of morphisms as above, such that their structure morphism $\lambda: \mathcal{Y} \to \mbox{spec}\,R$ is flat and $\lambda_*\scrO_{\mathcal{Y}} = \scrO_{\mbox{spec}\,R}$ holds universally. Furthermore, the generic fibre $\lambda_Q: \mathcal{Y}_Q \to \mbox{spec}\,Q$ is smooth.
\end{enumerate}
\end{defn}

\begin{remark}
If $\pi: \mathcal{Y} \to \frakX$ lies in $\mathcal{S}^{good}_{\frakX,D}$, then $\mathcal{Y}_Q$ is a geometrically connected, smooth and projective curve. Furthermore, every object $\pi_1: \mathcal{Y}_1 \to \frakX$ in $\mathcal{S}_{\frakX, D}$ is strictly dominated by an object $\pi_2: \mathcal{Y}_2 \to \frakX$ in $\mathcal{S}^{good}_{\frakX, D}$, i.e. there is a morphism $\phi: \mathcal{Y}_1 \to \mathcal{Y}_2$, which induces an isomorphism between the local rings of two generic points (see \cite{dw2}, Theorem 1). 
\end{remark}

\begin{defn}
\begin{enumerate} \item
Let $\frakX$ over $\overline{\bbZ}_p$ be a model of the smooth, connected projective curve $X$ over $\overline{\bbQ}_p$ and $D$ a divisor on $X$. The category $\mathfrak{B}_{\frakX_{\frako, D}}$ is defined as the full subcategory of vector bundles $\calE$ on $\frakX_{\frako}$ with the property: For all $n \geq 1$ there is a morphism $\pi: \mathcal{Y} \to \frakX$ in $\mathcal{S}_{\frakX, D}$, such that $\pi^*_n\calE_n$ is the trivial bundle on $\mathcal{Y}_n$, where $\pi_n, \mathcal{Y}_n$ and $\calE_n$ denote the reductions modulo $p^n$. \item
The full subcategory $\mathfrak{B}_{X_{\bbC_p,D}}$ of vector bundles on $X_{\bbC_p}$ consists of vector bundles $E$ that are isomorphic to $j^*\calE$ for some $\calE$ in $\mathfrak{B}_{\frakX_{\frako, D}}$, where $\frakX$ is a model of $X$ and $j: X_{\bbC_p} \into \frakX_{\frako}$ the open embedding of the generic fibre. \end{enumerate}
\end{defn}  

\begin{remark}
By \cite{dw2}, Theorem 17, we have $\mathfrak{B}^s_{X_{\bbC_p}} = \bigcup_D\mathfrak{B}_{X_{\bbC_p,D}}$.
\end{remark}

Now let $x \in X(\bbC_p)$ be a closed point and denote by $F_x$ the functor from the category of finite \'etale coverings of $X$ into the category of finite sets that maps such a covering $X' \to X$ onto the $\bbC_p$-valued points $x'$ over $x$. For two such $\bbC_p$-valued points $x',x \in X(\bbC_p)$ we call an isomorphism $F_x \To{\sim} F_{x'}$ an \'etale path from  $x$ to $x'$. The \'etale fundamental groupoid  $\Pi_1(X)$ is defined as the category whose objects are the $\bbC_p$-valued points of $X$, and for $x,x' \in X(\bbC_p)$ we set $\mbox{Mor}_{\Pi_1(X)}(x,x') := \mbox{Iso}(F_x,F_{x'})$. In particular, $\mbox{Mor}_{\Pi_1(X)}(x,x)$ is the \'etale fundamental group $\pi_1(X,x)$ of the curve $X$.
In \cite{dw2}, Deninger and Werner attach to each vector bundle $E$ in $\mathfrak{B}^{ps}_{X_{\bbC_p}}$ a continuous functor $\rho_E: \Pi_1(X) \to\mathbf{Vect}(\bbC_p)$, where $\mathbf{Vect}(\bbC_p)$ is the category of finite-dimensional vector spaces over $\bbC_p$. The category of these functors is denoted by $\mathbf{Rep}_{\Pi_1(X)}(\bbC_p)$.

\begin{thm} \label{hauptthm}
Let $X$ and $X'$ be smooth, connected and projective curves over $\overline{\bbQ}_p$ and $f: X \to X'$ a morphism.
\begin{enumerate} \item The functor $$\rho: \mathfrak{B}^{ps}_{X_{\bbC_p}} \longto \mathbf{Rep}_{\Pi_1(X)}(\bbC_p)$$is exact, $\bbC_p$-linear and additive. It commutes with tensor products, duals and inner homomorphisms. For $x \in X(\bbC_p)$, the natural functor $\omega_x: \mathfrak{B}^{ps}_{X_{\bbC_p}} \to \mathbf{Vect}(\bbC_p)$, which maps a vector bundle $E$ onto its fibre $E_x$, is faithful.  \item Pullback of vector bundles induces an exact and additive functor $f^*: \mathfrak{B}^{ps}_{X_{\bbC_p}} \to \mathfrak{B}^{ps}_{X'_{\bbC_p}}$ that commutes with tensor products, duals and inner homomorphisms. Further, there exists a canonical functor $$A(f): \mathbf{Rep}_{\Pi_1(X)}(\bbC_p) \longto \mathbf{Rep}_{\Pi_1(X')}(\bbC_p),$$such that the diagram  $$\xymatrix{ \mathfrak{B}^{ps}_{X_{\bbC_p}} \ar[r]^-{\rho} \ar[d]^{f^*} & \mathbf{Rep}_{\Pi_1(X)}(\bbC_p) \ar[d]_{A(f)} \\ \mathfrak{B}^{ps}_{X'_{\bbC_p}} \ar[r]^-{\rho} & \mathbf{Rep}_{\Pi_1(X)}(\bbC_p)}$$ is commutative. \item The categories $\mathfrak{B}^s_{X_{\bbC_p}}$ and $\mathfrak{B}^{ps}_{X_{\bbC_p}}$ are abelian categories.
\end{enumerate}   
\end{thm}

These statements can be found in \cite{dw2}, Theorem 33, Theorem 36, Proposition 35 and \cite{dw3}, Corollary 10. There is in particular a faithful functor $\mathfrak{B}^{ps}_{X_{\bbC_p}} \to \mathbf{Rep}_{\pi_1(X,x)}(\bbC_p)$ that satisfies all properties of this theorem. It follows from a more general result of Faltings (see \cite{fal}, \S5) that this functor is even fully faithful. We will later make use of this fact several times. Furthermore, the functor $\rho$ in the theorem is compatible with actions of the absolute Galois group $\mbox{Gal}_{\bbQ_p}$. Since this action is used to investigate the analytic monodromy group of a bundle in the last section of this paper, we describe it here in some detail, again following \cite{dw2}. For $\sigma \in \mbox{Gal}_{\bbQ_p}$ denote by $^{\sigma}\!X := X \tensor_{\overline{\bbQ}_p,\sigma}{\overline{\bbQ}_p}$, that is $^{\sigma}\!X$ is given by the cartesian diagram $$\xymatrix{^{\sigma}\!X \ar[r]^{\sim} \ar[d] & X \ar[d] \\ \mbox{spec}\,\overline{\bbQ}_p \ar[r]_{\sigma}^{\sim} & \mbox{spec}\,\overline{\bbQ}_p.}$$ For a closed point $x \in X$ let $^{\sigma}\!x \in {^{\sigma}\!X}$ be the image of $x$  with respect to the isomorphism $X \iso {^{\sigma}\!X}$. Likewise, one defines $^{\sigma}\!X_{\bbC_p}$ and $^{\sigma}\!E$ for a vector bundle $E$ on $X_{\bbC_p}$, since the Galois group also acts on $\bbC_p$. A morphism $f: E_1 \to E_2$ induces a morphism $$^{\sigma}\!f:= f \tensor_{\bbC_p, \sigma}\bbC_p :{^{\sigma}\!E_1} \longto {^{\sigma}\!E_2}.$$ One obtains a functor  $$\sigma_*: \calB^{ps}_{X_{\bbC_p}} \longto \calB^{ps}_{{^{\sigma}\!X}_{\bbC_p}}$$ that maps a vector bundle $E$ onto the vector bundle $^{\sigma}\!E$ and a morphism $f$ onto the morphism $^{\sigma}\!f$. We call a functor $F$ continuous, if for all objects $A,B$ the associated map $\mbox{Hom}(A,B) \to \mbox{Hom}(F(A),F(B))$ is a continuous map of topological spaces. For the category $\mathbf{Rep}_{\pi_1(X,x)}(\bbC_p)$ there is a continuous functor $$\mathbf{C}_{\sigma}: \mathbf{Rep}_{\pi_1(X,x)}(\bbC_p) \longto \mathbf{Rep}_{\pi_1({^{\sigma}\!X}, {^{\sigma}\!x})}(\bbC_p).$$ First, there is an isomorphism $\sigma_*: \pi_1(X,x) \To{\sim} \pi_1(^{\sigma}\!X, {^{\sigma}\!x})$ (see \cite{dw2}, below Proposition 24). Furthermore, one gets a continuous and $\sigma$-linear functor $\sigma_*: \mathbf{Vect}(\bbC_p) \to \mathbf{Vect}(\bbC_p)$, which maps the $\bbC_p$-vectorspace $V$ onto $^{\sigma}\!V := V \tensor_{\bbC_p,\sigma}\bbC_p$ and a morphism $f: V \to W$ onto $^{\sigma}\!f := f \tensor_{\bbC_p,\sigma}\bbC_p: {^{\sigma}\!V} \longto {^{\sigma}\!W}$. We then map a representation $\rho$ in the category $\mathbf{Rep}_{\pi_1(X,x)}(\bbC_p)$ onto the representation $\mathbf{C}_{\sigma}(\rho)$ defined by the commutative diagram $$\xymatrix{\pi_1(X,x) \ar[r]^{\rho} \ar[d]^{\sigma_*}_{\wr} & \mbox{GL}(V) \ar[d]^{\sigma_*}_{\wr} \\ \pi_1({^{\sigma}\!X},{^{\sigma}\!x}) \ar[r]^{\mathbf{C}_{\sigma}(\phi)} & \mbox{GL}({^{\sigma}\!V}).}$$ A morphism $f:V \to W$ of $\pi_1(X,x)$-modules is mapped onto the morphism $\mathbf{C}_{\sigma}(f):= {^{\sigma}\!f}: {^{\sigma}\!V} \to {^{\sigma}\!W}$. One easily sees that $\mathbf{C}_{\sigma}(f)$ is a morphism of $\pi_1({^{\sigma}\!X},{^{\sigma}\!x})$-modules.

\begin{remark} \label{basis}
Choosing a basis of the vector space $V$, one gets the commutative diagram $$\xymatrix{ \mbox{GL}(V) \ar[r]^{\sigma_*} \ar[d]^{\wr} & \mbox{GL}({^{\sigma}\!V}) \ar[d]_{\wr} \\ \mbox{GL}_r(\bbC_p) \ar[r]^{\sigma} & \mbox{GL}_r(\bbC_p),}$$ where the lower vertical morphism maps a matrix $A$ onto $\sigma(A)$. For an automorphism $f \in \mbox{GL}(V)$ with associated matrix $A = (a_{ij}) \in \mbox{GL}_r(\bbC_p)$, we have for the $i$-th basis vector $e_i$ of $V$ $$\sigma_*(f)(e_i) = (f \tensor_{\bbC_p, \sigma}\bbC_p)(e_i) = \sum_ja_{ij}e_j =\sum_j \sigma(a_{ij}) \cdot e_j,$$since $^{\sigma}\!V$ is the $\bbC_p$-vector space $V$ with scalar multiplication $\lambda \cdot v = \sigma^{-1}(\lambda)v$.
\end{remark}

There is the following connection between the Galois actions on these categories, see \cite{dw2}, Theorem 28:

\begin{thm} \label{galois}

There is a commutative diagram $$\xymatrix{\mathfrak{B}^{ps}_{X_{\bbC_p}} \ar[r]^-{\rho_E} \ar[d]^{\sigma_*} & \mathbf{Rep}_{\pi_1(X,x)}(\bbC_p) \ar[d]_{\mathbf{C}_{\sigma}} \\ \mathfrak{B}^{ps}_{{ ^{\sigma}\!X}_{\bbC_p}} \ar[r]^-{\rho_{{^\sigma}\!E}} & \mathbf{Rep}_{\pi_1({^{\sigma}\!X}, {^{\sigma}\!x})}(\bbC_p).}$$
\end{thm}  

Let $K \supseteq \bbQ_p$ be a field extension and $X$ a smooth, geometrically connected and projective curve over $K$. Further, let $E$ be a vector bundle on $X$, such that the vector bundle $E_{\bbC_p}$ lies in the category $\mathfrak{B}^{ps}_{X_{\bbC_p}}$. The scheme $^{\sigma}\!X_{\overline{\bbQ}_p}$ is defined by a cartesian diagram, and its universal property yields a commutative diagram $$\xymatrix{& X_{\bbC_p} \ar[rd] & \\ X_{\bbC_p} \ar[ru] \ar[rd] \ar[r]^{\sim} & {^{\sigma}\!X}_{\bbC_p} \ar[u]^{\wr} \ar[d] & \mbox{spec}\,\bbC_p \\ & \mbox{spec}\,\bbC_p. \ar[ur]_{\sigma} & }$$ The lower diagonal morphism is the structural morphism $X_{\bbC_p} \to \mbox{spec}\,\bbC_p$, and the upper one is the isomorphism $$X_{\bbC_p} = X \tensor_K\bbC_p \Longto{\mbox{id}\tensor \mbox{spec}\,\sigma} X \tensor \mbox{spec}\,\bbC_p = X_{\bbC_p}.$$ Hence we obtain a $\mbox{spec}\,\bbC_p$-isomorphism $X_{\bbC_p} \iso {^{\sigma}\!X}_{\bbC_p}$. Likewise, we identify $E_{\bbC_p} \iso {^{\sigma}}\!E_{\bbC_p}$. If $x \in X(K)$ is a $K$-valued point, Theorem \ref{galois} implies the following:

\begin{cor} \label{galo}
If $X$, $E$ and the closed point $x \in X$ are already defined over $K \supseteq \bbQ_p$, one obtains the commutative diagram $$\xymatrix{\mathfrak{B}^{ps}_{X_{\bbC_p}} \ar[r]^-{\rho_E} \ar[d]^{\sigma_*} & \mathbf{Rep}_{\pi_1(X,x)}(\bbC_p) \ar[d]_{\mathbf{C}_{\sigma}} \\ \mathfrak{B}^{ps}_{X_{\bbC_p}} \ar[r]^-{\rho_E} & \mathbf{Rep}_{\pi_1(X, x)}(\bbC_p).}$$  
\end{cor}

\section{Algebraic monodromy groups of vector bundles on p-adic curves} 

The category of semistable vector bundles of degree $0$ on the smooth, connected and projective curve $X_{\bbC_p}$ is abelian and possesses a natural fibre functor, namely the faithful functor $\omega_x: \mathfrak{B}^{ps}_{X_{\bbC_p}} \to \mathbf{Vect}(\bbC_p)$. In particular, it is a neutral Tannaka category, so it is equivalent to the category of finite-dimensional representations of an affine group scheme over $\bbC_p$. The latter is called the Tannaka dual group of this category. A good reference for this theory is \cite{dm}. We denote by $\mathfrak{B}_E$ the Tannaka subcategory generated by a vector bundle $E$ and by $G_E$ its Tannaka dual group. \medskip

Now let $E$ be a vector bundle with potentially strongly semistable reduction of degree $0$ on the smooth projective curve $X_{\bbC_p}$ and $\rho_E: \pi_1(X,x) \to \mbox{GL}(E_x)$ the associated continuous representation of the \'etale fundamental group. In this section we will investigate the Zariski closure of the image of this representation and the Tannaka dual group of this vector bundle. We show that both groups coincide under certain conditions and prove a connectedness criterion. Recall the following definitions for algebraic groups: \medskip

Let $G$ be a group scheme of finite type over a field $K$. The group scheme $G$ is called a linear algebraic group if there is a finite-dimensional $K$-vector space $V$ together with a closed immersion of group schemes $G \into \mathbf{G}\mathbf{L}_V$.
We denote by $\mathbf{Rep}_{G}(K)$ the category of finite-dimensional $G$-modules over $K$. A linear algebraic group $G$ is called linear reductive if all $G$-modules in $\mathbf{Rep}_{G}(K)$ are semisimple. If furthermore $G$ is smooth and connected, such that the unipotent radical $R_u(G)$ is trivial, then $G$ is called reductive. If $G$ is smooth and connected, with trivial radical $R(G)$, then $G$ is semisimple. A semisimple algebraic group $G$ with no nontrivial closed connected normal subgroups is called almost simple. \medskip

Note that a linear algebraic group over a field of characteristic $0$ is always smooth. There is the following well-known connection between linear reductive and reductive groups: 

\begin{prop} \label{reduktiv} 
If $K$ is a field of characteristic $0$ and $G$ a linear algebraic group over $K$, then the connected component $G^0$ is reductive if and only if the algebraic group $G$ is linear reductive. More generally, for $G$ an arbitrary affine group scheme the connected component $G^0$ is pro-reductive if and only if $G$ is linear reductive.
\end{prop}  
\begin{proof}
\cite{dm}, Remark 2.28. \end{proof}

\begin{prop}
Let $E$ be a semistable vector bundle of degree $0$ on the smooth, connected and projective curve $X_{\bbC_p}$. Its Tannaka dual group $G_E$ is a linear algebraic group. Denote by $E_x$ the fibre of $E$ in the point $x \in X(\bbC_p)$. Then there is a natural closed embedding $G_E \into \mathbf{G}\mathbf{L}_{E_x}$. 
Further, if $E$ is a polystable vector bundle, it follows that $G_E$ is linear reductive and hence that the connected component $G_E^0$ is a reductive algebraic group. 
\end{prop}
\begin{proof}
Due to \cite{dm}, Proposition 2.20b, the affine group scheme $G_E$ is linear algebraic if and only if the tensor category $\mathbf{Rep}_{G_E}(\bbC_p)$ is generated by some object of this category. Now $\mathbf{Rep}_{G_E}(\bbC_p)$ is of course generated by the $G_E$-module $E_x$, and it is easy to see that this generator corresponds to a faithful representation $G_E \into \mathbf{G}\mathbf{L}_{E_x}$. If the vector bundle $E$ is even polystable, then $\mathfrak{B}_E$ is semisimple, see for example \cite{dw3}, Theorem 12, so the last assertion follows from Proposition \ref{reduktiv}. 
\end{proof}

Let $E$ be a semistable vector bundle with potentially strongly semistable reduction and let $\rho_E: \pi_1(X,x) \to \mbox{GL}(E_x)$ be the associated representation of the fundamental group. We denote by $G_{\rho_E}$ the image of this representation and by $\overline{G}_{\rho_E}$ its Zariski closure in $\mathbf{GL}_{E_x}$. 

\begin{lemma} \label{algmono}
The natural functor $\mathbf{Rep}_{\overline{G}_{\rho_E}}(\bbC_p) \to \mathbf{Rep}_{G_{\rho_E}}(\bbC_p)$ is fully faithful.
\end{lemma}

\begin{proof}
Since for any group $G$ and finite-dimensional $G$-modules $V$ and $W$ we have $\mbox{Hom}_G(V,W) = (V^*\tensor W)^G$, it suffices to show that $V^{\overline{G}_{\rho_E}} = V^{G_{\rho_E}}$ holds for all $\overline{G}_{\rho_E}$-modules $V$. Let $v \in V$ be a nontrivial $G_{\rho_E}$-invariant element. Denoting the image of the representation $G_{\rho_E} \to \mbox{GL}(V)$ by $G_V$, we have $$\mbox{im}(\overline{G}_{\rho_E} \longto \mathbf{GL}_V) = \overline{G}_V$$ and we furthermore may assume that $$G_V \subseteq \left( \begin{array}{cccc} 1 & & & \\ 0 & & * & \\ \vdots & & * & \\ 0 & & & \end{array} \right) \subseteq  \mbox{GL}(V).$$ The group consisting of these matrices is Zariski closed in $\mbox{GL}(V)$, hence $v$ is also a $\overline{G}_{\rho_E}$-invariant element. The other inclusion is clear.
\end{proof}  

If we consider $\mathbf{Rep}_{G_{\rho_E}}(\bbC_p)$ as a full subcategory of $\mathbf{Rep}_{\pi_1(X,x)}(\bbC_p)$, we have fully faithful functors $$\mathbf{Rep}_{G_E}(\bbC_p) \iso \mathfrak{B}_E \longto \mathbf{Rep}_{G_{\rho_E}}(\bbC_p) \longleftarrow \mathbf{Rep}_{\overline{G}_{\rho_E}}(\bbC_p),$$ compatible with tensor products and duals. They furthermore commute with the natural fibre functor. Since the category on the left is contained in $\mathbf{Rep}_{\overline{G}_{\rho_E}}(\bbC_p)$ if considered as full subcategories of $\mathbf{Rep}_{G_{\rho_E}}(\bbC_p)$, we obtain a fully faithful functor $\mathbf{Rep}_{G_E}(\bbC_p) \to \mathbf{Rep}_{\overline{G}_{\rho_E}}(\bbC_p)$. It follows from \cite{dm}, Proposition 2.21b that this corresponds to a closed immersion $$\overline{G}_{\rho_E} \longinto G_E$$because $\rho_E: G_{\rho_E} \to \mbox{GL}(E_x)$ is also a tensor generator of $\mathbf{Rep}_{\overline{G}_{\rho_E}}(\bbC_p)$. It is quite natural to conjecture that for polystable vector bundles this is always an isomorphism. We will show below that this is true if the representation $\rho_E$ is semisimple, see also the remark in \cite{dw3} on the last page. \medskip

Let $G$ be a linear reductive group with a faithful $G$-module $V$ and denote by $T^{r,s}(V)$ for all $r,s \in \bbN$ the $G$-module $V^{\tensor r} \tensor (V^*)^{\tensor s}$. For an algebraic subgroup $H \subseteq G$ let $H'$ be the biggest subgroup of $G$ with $T^{r,s}(V)^H = T^{r,s}(V)^{H'}$ for all $r,s \in \bbN$. We surely have $H \subseteq H'$, and if $H$ itself is linear reductive this is even an equality: 

\begin{prop}  \label{schnitte}
Let $G \into \mathbf{G}\mathbf{L}_V$ be a linear reductive group over a field $K$ of characteristic $0$ and let $H \subseteq G$ be a linear reductive subgroup. Then $H = H'$.
\end{prop}
\begin{proof} \cite{dm}, Proposition 3.1. \end{proof}

Let $E$ be a polystable vector bundle of degree $0$ and let $F$ be a vector bundle lying in the Tannaka category $\mathfrak{B}_E$. The fully faithfulness of the functor $\mathfrak{B}_E \to \mathbf{Rep}_{G_E}(\bbC_p)$ yields the equations 
$$\Gamma(X_{\bbC_p}, F) = \mbox{Hom}(\scrO_{X_{\bbC_p}}, F) = \mbox{Hom}_{G_E}(\bbC_p, F_x) = F_x^{G_E}.$$
Here $\bbC_p$ denotes the trivial $G_E$-module. The global sections of the bundle $F$ correspond to the $G_E$-invariants of its fibre $F_x$.
If we apply the proposition above with $G := \mathbf{GL}_{E_x}$, $H := G_E$, we see that $G_E$ depends only on the global sections $\Gamma(X_{\bbC_p}, T^{r,s}(E))$ with $r,n \in \bbN$. If we denote by $\gamma_x$ the image of the global section $\gamma \in \Gamma(X_{\bbC_p}, T^{r,s}(E))$ in the fibre $T^{r,s}(E_x)$, the following corollary holds:

\begin{cor} \label{corschnitt}
For all polystable vector bundles $E$ of degree zero we have  $$G_E(\bbC_p) = \left.\{g \in \mbox{GL}(E_x)\;|\; {g\gamma_x = \gamma_x \mbox{ for all } \gamma \in \Gamma(X_{\bbC_p}, T^{r,s}(E)) \atop \mbox{ and all } r,s \in \bbN}\right\}.$$ 
\end{cor}

\begin{proof}
The linear reductive group $G_E$ is, due to Proposition \ref{schnitte}, the biggest algebraic subgroup of $\mathbf{G}\mathbf{L}_{E_x}$ fixing all global sections $\Gamma(X_{\bbC_p}, T^{r,s}(E)) \subseteq T^{r,s}(E_x)$ for all $r,n \in \bbN$.
\end{proof}

\begin{prop} \label{gleich}
Let $E$ be a polystable vector bundle with potentially strongly semistable reduction of degree $0$ and denote by $\rho_E: \pi_1(X,x) \to \mbox{GL}(E_x)$ the corresponding representation. If $\rho_E$ is semisimple, we have $$\overline{G}_{\rho_E} = G_E.$$
\end{prop}

\begin{proof}
If $E_x$ is a semisimple $G_{\rho_E}$-module, then it is also semisimple as $\overline{G}_{\rho_E}$-module. Because of char($\bbC_p) = 0$, all objects $T^{r,s}(E_x)$ are semisimple $\overline{G}_{\rho_E}$-modules, so  $\overline{G}_{\rho_E}$ is linear reductive, since $\mathbf{Rep}_{\overline{G}_{\rho_E}}(\bbC_p)$ is generated by the module $E_x$. Because of $$T^{r,s}(E_x)^{\overline{G}_{\rho_E}} = T^{r,s}(E_x)^{\pi_1(X,x)} = T^{r,s}(E_x)^{G_E}$$ for all $r,s \in \bbN$ the result follows from Proposition \ref{schnitte}. 
\end{proof}

We now want to study the connected component of the identity of the algebraic monodromy groups $G_E$ and $\overline{G}_{\rho_E}$. The functor $$F: FEt/X \longto \mathbf{Sets},\; F(Y):= \mbox{Hom}_X(x,Y)$$ from the category of finite \'etale coverings of $X$ to the category of sets gives an equivalence between $FEt/X$ and the category of finite $\pi_1(X,x)$-sets with a continuous action, see \cite{sga1}, Expos\'e V.7. Recall that the image of the fundamental group lies Zariski dense in $G_E$. For the finite group $H:=\overline{G}_{\rho_E}/\overline{G}_{\rho_E}^0(\bbC_p)$ with a continuous surjective morphism $\pi(X,x) \onto H$ we hence find a finite \'etale covering $f_0: Y \to X$ and a short exact sequence $$1 \longto \pi_1(Y,y) \longto \pi_1(X,x) \longto H \longto 1.$$ For the vector bundle $f_0^*E$ on $Y_{\bbC_p}$ the equation $G_{\rho_{f_0^*E}}= G_{\rho_E}^0$ holds due to the compatibility of the pullback with the representation, see Theorem \ref{hauptthm}. Let us denote this pullback bundle by $E_0$. This shows:

\begin{prop} \label{uber}
Let $E$ be a semistable vector bundle with potentially strongly semistable reduction of degree $0$. For the vector bundle $E_0$ with respect to the finite \'etale covering  $f_0 : Y \to X$ we have $$\overline{G}_{\rho_{E_0}} = \overline{G}_{\rho_E}^0,\; G_{E_0}= G_E^0 \mbox{ and } \mbox{Aut}_X(Y) = \overline{G}_{\rho_{E_0}}/\overline{G}_{\rho_E}^0(\bbC_p) = G_E/G_E^0(\bbC_p).$$
\end{prop}

\begin{cor} \label{stable}
Let $E$ be a stable vector bundle with potentially strongly semistable reduction of degree $0$, such that the Tannaka dual group $G_E$ is connected. Then for every finite morphism $f: Y \to X$ of smooth, projective curves the pullback bundle $f_{\bbC_p}^*(E)$ is also stable.
\end{cor}

\begin{proof}
It is known that the image of the canonical map $\pi_1(Y,y) \to \pi_1(X,x)$ has finite index in $\pi_1(X,x)$. To see this, choose a divisor $D$ on $X$ such that $f|_V: V:= Y\backslash f^{-1}(D) \to X \backslash D =: U$ is \'etale. Then $\pi_1(V,y) \into \pi_1(U,x)$ has finite index and $\pi_1(U,x) \onto \pi_1(X,x)$, $\pi_1(V,y) \onto \pi_1(Y,y)$ are surjective for curves, which at once yields the claim. Then $G_{\rho_{f^*(E)}} \subseteq G_{\rho_{E}}$ also has finite index, and since $G_{\rho_{E}}$ is connected, these algebraic groups have to coincide, so the pullback bundle $f^*(E)$ is stable.
\end{proof}

\begin{prop} \label{halbeinfach}
Let $E$ be a stable vector bundle of rank $\geq 2$ with potentially strongly semistable reduction of degree $0$. If $\det(E)$ has finite order and if $E_0$ is again a stable bundle, then $G_E^0$ is a semisimple algebraic group and the fibre $E_x$ is an irreducible $G_E^0$-module.
\end{prop}

\begin{proof}  
We have seen in Proposition \ref{uber} that the Tannaka dual group of the bundle $E_0$ is the connected algebraic group $G_E^0$. The assumptions then imply that the fibre $E_x$ is an irreducible $G_E^0$-module. Since $G_E^0$ is a reductive algebraic group, it satisfies $R(G_E^0) = Z(G_E^0)^0$, where $Z(G_E^0)$ denotes the center of $G_E^0$. It remains to show that $Z(G_E^0)$ is a finite group. We have $Z(G_E^0)(\bbC_p) \subseteq \mbox{End}_{G_E^0}(E_x) \cap \mbox{SL}(E_x)$, because $Z(G_E^0)(\bbC_p)$ commutes with $G_E^0(\bbC_p)$, so every element in $Z(G_E^0)(\bbC_p)$ defines a $G_E^0(\bbC_p)$-endomorphism of $E_x$. On the other hand $\mbox{det}^n(E) \iso \scrO_{X_{\bbC_p}}$ for some $n \in \bbN$ yields $G_E^0 \subseteq \mathbf{SL}_{E_x}$. But then the Lemma of Schur implies $Z(G_E^0)(\bbC_p) \subseteq \lambda\mbox{id}_r \cap \mathbf{SL}(E_x)$ with $\lambda \in \bbC_p$, hence $Z(G_E^0)(\bbC_p)$ is finite. 
\end{proof}

For $q \in \bbQ^+$ denote by $\frako_q$ the ideal $\frako/p^q\frako$, where $p^q$ is a zero of $X^b-p^a$ in $\frako$, with $q = a/b$ for positive integers $a,b$. Observe that this ideal is independent of the choice of the zero, since they all have the same absolute value.  

\begin{prop} \label{prop2}
Let $X$ be a smooth, connected and projective curve over $\overline{\bbQ}_p$ and let $E$ be a vector bundle of rank $r$ on $X_{\bbC_p}$. Denote by $\frakX$ a model of $X$ and by $\calE$ a vector bundle on $\frakX_{\frako}$ with generic fibre $E$, such that the vector bundle $\calE_q$ is trivial for some $q \in \bbQ^+$. In particular, the bundle $E$ has strongly semistable reduction and the image of the representation $\rho_E: \pi_1(X,x) \to \mbox{GL}_r(E_x)$ lies (possibly after a change of basis) in $\mbox{GL}_r(\frako)$ and is trivial modulo $p^q\frako$.
\end{prop} 

\begin{proof}
Let $(q_n)_{n \in \bbN}$ be a monotonously growing and divergent sequence of positive rational numbers. A vector bundle $\calE$ on $\frakX_{\frako}$ lies in the category $\mathfrak{B}_{\frakX_{\frako},D}$ if and only if for all $n \in \bbN$ there is a morphism $\pi: \mathcal{Y} \to \frakX$ in $\mathcal{S}^{good}_{\frakX,D}$ such that the reduction  $\pi_{q_n}^*\calE_{q_n}$ on $\mathcal{Y}_{q_n} := \mathcal{Y} \tensor_{\overline{\bbZ}_p} \frako/p^{q_n}\frako$ is the  trivial bundle. Then one can define continuous morphisms $\rho_{\calE,n}: \pi_1(X\backslash D,x) \to \mbox{Aut}_{\frako_{q_n}}(\calE_{x_{q_n}})$ as in \cite{dw2}, page 579f, and one obtains a continuous morphism $\projlim_{n \in \bbN}\rho_{\calE,n} = \rho_{\calE}: \pi_1(X\backslash D,x) \to \mbox{Aut}_{\frako}(\calE_{x_{\frako}})$. We have to show that this construction is independent from a choice of the sequence $(q_n)$. Denote by $(q'_n)$ another sequence as above. Since the projective limit does not change if one takes subsequences, we may assume $q_n \leq q'_n$ for all $n \in \bbN$. Consider for $\gamma \in \pi_1(X\backslash D,x)$ and $y \in \mathcal{Y}(\bbC_p)$ the commutative diagram $$\xymatrix{ \mbox{spec }\frako_{q_n} \ar[d]^-a \ar[r]^-{y_{q_n}} & \mathcal{Y}_{q_n} \ar[d]^-b & \ar[l]_-{\gamma y_{q_n}} \mbox{spec } \frako_{q_n} \ar[d]^-a \\ \mbox{spec }\frako_{q'_n} \ar[r]^-{y_{q'_n}} & \mathcal{Y}_{q'_n} & \ar[l]_-{\gamma y_{q_n}} \mbox{spec }\frako_{q'_n}}$$ with canonical vertical morphisms denoted by $a$ und $b$.  This yields at once the commutative diagram $$\xymatrix{\calE_{x_{q'_n}} \ar[d]^{a^*} & \Gamma(\mathcal{Y}_{q'_n},\pi^*_{q'_n}\calE_{q'_n}) \ar[l]^-{\sim}_-{y^*_{q'_n}} \ar[d]^{b^*} \ar[r]_-{\sim}^-{(\gamma y)^*_{q'_n}} & \calE_{x_{q'_n}} \ar[d]^{a^*} \\ \calE_{x_{q_n}} & \Gamma(\mathcal{Y}_{q_n}, \pi^*_{q_n}\calE_{q_n}) \ar[l]^-{\sim}_-{y^*_{q_n}} \ar[r]_-{\sim}^-{(\gamma y)^*_{q_n}} & \calE_{x_{q_n}}.}$$ In particular, the $\pi_1(X\backslash D,x)$-action on $\calE_{x_{q_n}}$ with respect to the canonical morphism $a^*$ is compatible with the action of this group on $\calE_{x_{q'_n}}$, because due to \cite{dw2}, page 579 we have $\rho_{\calE,n}(\gamma) = (\gamma y_{q_n})^* \circ (y^*_{q_n})^{-1}$ respectively $\rho'_{\calE,n}(\gamma) = (\gamma y_{q'_n})^* \circ (y^*_{q'_n})^{-1}$. Hence the commutative diagram $$\xymatrix{\projlim\limits_{n \in \bbN}\calE_{x_{q'_n}} \ar[d]^{\wr} \ar[r] & \projlim\limits_{n \in \bbN}\calE_{x_{q_n}} \ar[d]^{\wr} \\ \calE_{x_{\frako}} \ar[r]^{\mbox{id}} & \calE_{x_{\frako}}}$$ consists of $\pi_1(X\backslash D,x)$-equivariant morphisms, where the group action on $\mathcal{E}_{x_{\frako}}$ on the left is given by $\rho'_{\calE}$ and on the right by $\rho_{\calE}$.\medskip

Let us now consider a vector bundle $\calE$ as in Proposition \ref{prop2}.     
It lies in the category $\mathfrak{B}_{\frakX_{\frako},D}$ due to \cite{dw2}, Theorem 16, because its special fibre $\calE_{\overline{\bbF}_p}$ is the trivial bundle. It follows that its generic fibre $E$ lies in $\mathfrak{B}^{s}_{X_{\bbC_p}} = \bigcup_D\mathfrak{B}_{X_{\bbC_p},D}.$ One can explicitly describe the corresponding representation $\rho_E: \pi_1(X,x) \to \mbox{GL}(E_x)$, see \cite{dw2}, page 587:
Denote by $j_{\frakX_{\frako}}: X_{\bbC_p} \into \frakX_{\frako}$ the canonical embedding and let $\psi: E \To{\sim}j^*_{\frakX_{\frako}}\calE$ be some isomorphism of vector bundles on the curve $X_{\bbC_p}$. Furthermore, let  $$\psi_x: E_x \Longto{\sim} (j^*_{\frakX_{\frako}}\calE)_x = \calE \tensor_{\frako}\bbC_p$$be the corresponding morphism of the fibres. For all $\gamma \in \pi_1(X\backslash D,x)$ we then get a representation $\rho_{E,D} : \pi_1(X\backslash D) \to \mbox{GL}(E_x)$ given by $$\rho_{E,D}(\gamma) = \psi_x^{-1} \circ (\rho_{\calE}(\gamma) \tensor_{\frako}\bbC_p) \circ \psi_x.$$ The representation $\rho_E$ is induced by the surjective morphism $\pi_1(X\backslash D,x) \onto \pi_1(X,x)$ together with the representation $\rho_{E,D}$, see \cite{dw2}, Proposition 35. It remains to show that $\rho_{\calE}(\gamma)$ becomes trivial modulo $p^q\frako$. The identity morphism $\mbox{id}: \frakX \to \frakX$ is strictly dominated by some morphism $\mathcal{Y} \To{\pi} \frakX$ in $\mathcal{S}^{good}_{\frakX,D}$ (\cite{dw2}, Theorem 1), i.e. there exists a commutative diagram $$\xymatrix{\mathcal{Y} \ar[rr]^{\varphi} \ar[rd]^{\pi} & & \frakX \ar[dl]^{\mbox{id}} \\ & \frakX & }$$ such that the induced morphism of the generic fibres $$Y := \mathcal{Y} \tensor_{\overline{\bbZ}_p} \overline{\bbQ}_p \longto X$$ is an isomorphism, because $Y$ is a smooth, connected and projective curve due to the definition of the category $\mathcal{S}^{good}_{\frakX, D}$ (\cite{dw2}, page 556). Hence the pullback $\pi^*_q \calE_q$ is the trivial bundle, since  $\calE_q$ is trivial. The representation $\rho_{\calE}$ modulo $p^q\frako$ =: $\rho_{\calE, 1}: \pi_1(X\backslash D,x) \to \mbox{GL}_r(\frako_q)$ now factors over the group $\mbox{Aut}_{X\backslash D}(Y\backslash \pi^*_{\overline{\bbQ}_p}D)= \{1\}.$  
\end{proof}  

The following Lemma is needed for the proof of the main result of this section.

\begin{lemma} \label{lemma1}
Let $\mathfrak{m}$ denote the maximal ideal of the valuation ring $\frako$ of $\bbC_p$. For 
$\lambda \in \frako$ and $M \in M_r(\mathfrak{m})$ an $(r \times r)$-matrix with entries in $\mathfrak{m}$ we have \begin{enumerate} \item If $1- \lambda$ is a primitive $l^n$th root of unity, then $|\lambda| = 1$ holds for $ l \neq p$ and $|\lambda| = p^{-\frac{n}{p^n-1}}$ holds for $l = p$.\item If $1-\lambda$ is an eigenvalue of $1_r - M$, we have the inequality $|\lambda| \leq |M|$. \end{enumerate} \end{lemma}  

\begin{proof}
\begin{enumerate} 
\item The primitive $l^n$-th roots of unity are not equal to $1$ in the residue field $\frako/\mathfrak{m} = \overline{\mathbb{F}}_p$ if $l \neq p$, so the first claim follows at once. For $l = p$, the roots of unity are trivial in $\frako/\mathfrak{m}$, therefore $\lambda \in \frako$ holds. Furthermore, because of $(1-\lambda)^{p^n} - 1 = 0$, we have the equation $\sum_{i = 1}^{p^n}{p^n \choose i}(-\lambda)^i = 0$. This can only be true if the first and the last term have the same absolute value, i.e.  $|\lambda|^{p^n} = |p^n\lambda|$, because the absolute values of all the other terms are strictly smaller than $|p^n\lambda|$. Hence we have $|\lambda| = p^{-\frac{n}{p^n-1}}$. \item The characteristic polynomial of the matrix $M$ is $$p(x) = \mbox{det}(M-x\cdot 1_r) = (-x)^r + \sum_{i=0}^{r-1}m_ix^i \;\; \mbox{with }|m_i|\leq |M|^{r-i}.$$  
It is obvious that a zero $\lambda$ of the polynomial $p$ satisfies the inequality $|\lambda|\leq |M|$, otherwise the first term would be strictly greater than all the other terms. The claim now follows from the fact that $p(1-x)$ is the characteristic polynomial of the matrix $1_r-M$. \end{enumerate} \end{proof}

An algebraic group $G$ is not connected if and only if there exists a non-trivial finite algebraic group $H$ together with a surjection $G \onto H$. For any such surjection, there is a finite-dimensional vector space $V$ and a morphism $G \to \mathbf{GL}_V$ that factors over a closed immersion $H \into \mathbf{GL}_V$, see \cite{hum2}, Theorem 11.5. Hence an algebraic group is connected if and only if there is no non-trivial finite-dimensional representation of $G$ with finite image.

\begin{thm} \label{zusammen}

Let $E$ be a vector bundle of rank $r$ over the smooth, connected and projective curve $X_{\bbC_p}$. If $E$ has the properties described in Proposition \ref{prop2}, with $q > \frac{1}{p-1}$ for odd primes $p$ and $q \geq 1$ for $p=2$, then the algebraic monodromy groups $G_E$ and $\overline{G}_{\rho_E}$ are connected.
\end{thm}

\begin{proof}
Because of Proposition \ref{prop2} we can assume that the corresponding representation $\rho_E: \pi_1(X,x) \to \mathrm{GL}_r(\frako)$ is trivial modulo $p^q\frako$. Consider a matrix $N \in G_{\rho_E} \subseteq \mathrm{GL}_r(\frako)$, so $N = 1_r - M$ with $M \in M_r(p^q\frako)$. The eigenvalues $1- \lambda_i$ of $N$ fulfill the equation $|\lambda_i| \leq p^{-q}$ by Lemma \ref{lemma1}. The set of the eigenvalues of all matrices $N \in G_{\rho_E}$ will be called the eigenvalues of this representation. They are of course independent of a choice of a basis of the vectorspace $E_x$.
Now we consider the canonical functor $$\mathbf{Rep}_{\pi_1(X,x)}(\frako) \longto \mathbf{Rep}_{\pi_1(X,x)}(\frako/p^q\frako),$$which maps a $\pi_1(X,x)$-module $V$ to a module $V \tensor \frako/p^q\frako$ and a morphism $f: V \to W$ to a morphism $f': V \tensor \frako/p^q\frako \to W\tensor \frako/p^q\frako$. It is easy to see that this functor is compatible with duals, tensor products and direct sums. In particular, the eigenvalues of the representations of the subcategory generated by the representation $\rho_E$, i.e. subquotients of direct sums of some $T^{r,s}(E_x)$, with $r,s \in \bbN$, are trivial modulo $p^q$.  
It suffices to show that for all nontrivial $G_E$-modules $V$ the image of the associated representation $\rho: G_E \to \mathbf{GL}_V$ is not finite. Denote by $G_V$ the image of this representation. It follows from Proposition \ref{gleich} that it contains a Zariski dense subgroup that is trivial modulo $p^q$. Let $N \in G_V(\bbC_p)$ be a nontrivial element. 
There are now two possibilities: Either $N$ has a nontrivial eigenvalue, or all eigenvalues are equal to $1$ and $N$ contains a Jordan block of size $> 1$. In the first case the eigenvalue is trivial $\mbox{ mod }p^q$, hence it is not a root of unity because of Lemma \ref{lemma1}. So the group generated by $N$ can not be a finite group. This is also true in the second case.
\end{proof}

In the sequel we will call a vector bundle $E$ on $X_{\bbC_p}$ trivial modulo $p^q$ if it satisfies the assumptions of Theorem \ref{zusammen}. 

\begin{cor}
Let $E$ be a stable vector bundle of degree $0$ on the smooth, connected and projective curve $X_{\bbC_p}$. If $E$ is trivial modulo $p^q$ and $f: Y \to X$ a finite morphism of smooth and projective curves over $\overline{\bbQ}_p$, then the pullback bundle $f^*_{\bbC_p}(E)$ is also stable. \end{cor}

\begin{proof}
Follows at once from Corollary \ref{stable}. 
\end{proof}

We finish this section with a nice criterion by Larsen (see for example \cite{kat1}, Theorem 1.1.6) to check whether the Tannaka dual group of a polystable vector bundle $E$ of degree $0$ and rank $r$ is either finite or one of the classical groups $\mathbf{S}\mathbf{L}_r$, $\mathbf{S}\mathbf{O}_r$ or $\mathbf{S}\mathbf{p}_r$. This works for arbitrary smooth projective varieties over an algebraically closed field of characteristic $0$. Note that the category of semistable vector bundles of degree $0$ is not abelian in general, so it cannot be a neutral Tannaka category. But the subcategory of polystable bundles is neutral Tannakian. 

\begin{prop} \label{larsen} 
Let $E$ be a polystable vector bundle of degree $0$ and rank $r$ on the smooth projective variety $X$ over an algebraically closed field $K$ of characteristic $0$ and denote by $F$ the vector bundle $\mathcal{E}nd(\mathcal{E}nd(E)) = E \tensor E \tensor E^* \tensor E^*$. \begin{enumerate} \item If $\dim_{\bbC_p}(\Gamma(X_{\bbC_p},F)) = 2$, then $G_E \supseteq \mathbf{SL}_{E_x}$ or $G_E/(G_E \cap \mbox{ skalars})$ is a finite algebraic group. If $\det(E)$ is of finite order, we have in particular $G_E^0 = \mathbf{SL}_{E_x}$ or $G_E$ is a finite algebraic group. \item If $\Gamma(X_{\bbC_p}, S^2(E)) \neq 0$ and $\dim_{\bbC_p}(\Gamma(X_{\bbC_p}, F)) = 3$, then  $G_E = \mathbf{O}_{E_x},\; G_E = \mathbf{SO}_{E_x}$ or $G_E$ is a finite algebraic group. \item If $r \geq 3$, $\Gamma(X_{\bbC_p}, \bigwedge^2(E)) \neq 0$ and $\dim_{\bbC_p}(\Gamma(X_{\bbC_p}, F)) = 3$, it follows that $G_E = \mathbf{Sp}_{E_x}$ or that $G_E$ is a finite algebraic group.
\end{enumerate} \end{prop}

\begin{proof}
This follows immediately from the equivalence of neutral Tannakian categories $\mathfrak{B}_E \To{\sim} \mathbf{Rep}_{G_E}(K)$ and the moment criterion for representations of algebraic groups by Larsen (\cite{kat1}, Theorem 1.1.6).
\end{proof}

\begin{remark}
If $E$ is not a finite vector bundle of rank $r$, for example a nontrivial vector bundle on the smooth, connected and projective curve $X_{\bbC_p}$ with trivial reduction modulo $p^q$ for some rational $q$ as in Theorem \ref{zusammen}, then its Tannaka dual group $G_E$ is not finite and Proposition \ref{larsen} yields a criterion for $G_E$ being one of the classical groups $\mathbf{S}\mathbf{L}_r$, $\mathbf{S}\mathbf{O}_r$ or $\mathbf{S}\mathbf{p}_r$.    
\end{remark}

\section{The Tannaka dual group of kernel bundles}

In this section $K$ denotes some algebraically closed field of characteristic $0$. By a kernel bundle on the smooth projective variety $X$ with polarization $\scrO(1)$ we mean a vector bundle $E$ on $X$ sitting in a short exact sequence $$0 \longto E \longto \bigoplus_{i=1}^c \scrO(a_i) \longto \bigoplus_{j=1}^d\scrO(b_j) \longto 0$$ with $a_i,b_j \in \bbZ$. If $d = 1$, such a bundle is called a syzygy bundle.

There are several restriction theorems for torsion-free sheaves on projective varieties, for example by Bogomolov, Flenner, or Mehta and Ramanathan. A good reference is the book \cite{huy}, II.7. Recently, A. Langer proved a very strong restriction theorem, which we will apply to certain stable kernel bundles on the projective space. Finally, we will be able to compute the type of the Tannaka dual group for the restriction of some of these bundles to smooth curves of sufficiently high degree. Let $H$ be a very ample divisor corresponding to the line bundle $\scrO(1)$ on the smooth projective variety $X$. Recall the following generalization of semistability for torsion free sheaves $E$ of rank $r$ on $X$: Define $\mu(E) := c_1(E)H^{n-1}/r$ and call $E$ slope $H$-semistable if for all proper subsheaves $F \subset E$ the inequality $\mu(F) \leq \mu(E)$ holds (and stable if it is strictly smaller).
We denote by $R$ the natural number ${r \choose l}{r-2 \choose l-1}$ for $r \geq 2, l = [\frac{r}{2}]$ and by $\Delta(E) := 2rc_2(E)-(r-1)c_1^2(E)$ the discriminant of a vector bundle $E$ with chern classes $c_1(E)$ and $c_2(E)$.

\begin{thm}[Langer, \cite{langer}, Theorem 5.2] \label{langer}
 Let $E$ be a slope $H$-stable torsion free sheaf on a smooth, projective variety $X$ of dimension $n$ over $K$ and let $a$ be an integer such that $$a > \frac{r-1}{r}\Delta(E)H^{n-2}+\frac{1}{r(r-1)H^n}.$$Then for every normal divisor $D \in |aH|$ such that the restriction $E_D$ is torsion free, we have that $E_D$ is also slope $H_D$-stable.
\end{thm}

\begin{remark} \label{chern}
Since we will consider kernel bundles on the projective space, their chern classes can be easily computed. If $E$ is a kernel bundle on $\bbP^n_K$ sitting in the short exact sequence $0 \to E \to \bigoplus_{i=1}^c \scrO(a_i) \to \bigoplus_{j=1}^d\scrO(b_j) \to 0$, we have for the total chern class $c(\bigoplus_{i=1}^c \scrO(a_i)) = c(E)c(\bigoplus_{j=1}^d\scrO(b_j))$, hence $$c(E) = \frac{c(\bigoplus_{i=1}^c \scrO(a_i))}{c(\bigoplus_{j=1}^d\scrO(b_j))} = \prod_{i = 1}^{c}(1+a_ih)/\prod_{j=1}^{d}(1+b_jh),$$ where the Chow ring is $A(\bbP^n_K) \iso \bbZ[h]/h^{n+1}$ with $h$ the class of a hyperplane. The coefficient of $h^i$ is the ith chern class with respect to this identification.

%Die Chernklassen eines Syzygienbndels $$E:=\mbox{Syz}(f_1,\dots,f_{r+1})(0),$$ das von Polynomen $f_1,\dots,f_{r+1}$ vom Grad $d_1,\dots,d_{r+1}$ erzeugt wird, lassen sich leicht berechnen. Wegen der kurzen exakten Sequenz $$0 \longto E \longto \bigoplus_{i = 1}^{r+1} \scrO(-d_i) \longto \scrO \longto 0$$ gilt $c(\bigoplus_{i=1}^{r+1}\scrO(-d_i)) = c(E)c(\scrO)$, d.h. $$c(E) = \frac{c(\bigoplus_{i = 1}^{r+1} \scrO(-d_i))}{c(\scrO)} = \prod_{i = 1}^{r+1}(1-d_ih) = 1-\sum_{i = 1}^{r+1}d_ih + \frac{(\sum_{i=1}^{r+1} d_i)^2-\sum_{i=1}^{r+1} d_i^2}{2}h^2$$ bezglich des Isomorphismus 
\end{remark}

\begin{thm}[Bohnhorst-Spindler, \cite{boh}, Theorem 2.7.] \label{bohn}
Let $E$ be a vector bundle of rank $r$ on the projective space $\bbP^r_K$ with a resolution $$0 \longto \bigoplus_{i = 1}^c \scrO(a_i) \longto \bigoplus_{j = 1}^{c+r}\scrO(b_j) \longto E \longto 0$$ where $a_1\geq a_2\geq\dots\geq a_c$, $b_1\geq b_2\dots\geq b_{c+r}$ and $a_i < b_{c+i}$ for $i = 1,\dots,r$. Equivalent are: \begin{enumerate} \item 
The bundle $E$ is stable. \item The inequality $b_1 < \mu(E) = \frac{1}{r}(\sum_{j = 1}^{c+r}b_j - \sum_{i = 1}^ca_i)$ holds. \end{enumerate}
\end{thm}

\begin{lemma} \label{nkleinerr}
For any stable vector bundle $E$ of rank $r$ and degree $0$ on the projective space $\bbP^r_{K}$ with a resolution as in Theorem \ref{bohn} we have $$H^k(\bbP^r_{K}, E^{\tensor n}(m)) = 0 \mbox{ for }m \leq 0, 0 < n \leq r-1, k < r-n.$$ In particular, $\Gamma(\bbP^r_{K}, E^{\tensor n}) = 0$ for all $0< n \leq r-1$.
\end{lemma}    

\begin{proof}
One just has to compute the cohomology of the sheaves $\scrO(m)$. We have $H^k(\bbP^r_{K}, \scrO(m)) = 0$ for all $0 < k < r$ and all $m$, and $H^0(\bbP^r_{K}, \scrO(m)) = 0$ for $m<0$ (\cite{har}, Theorem 5.1). The tensor product of a short exact sequence of vector bundles with a vector bundle is still a short exact sequence, hence $$0 \longto \bigoplus_{i = 1}^cE^{\tensor n-1}(m+a_i) \longto \bigoplus_{j = 1}^{c+r}E^{\tensor n-1}(m+b_j) \longto E^{\tensor n}(m) \longto 0$$ is exact. Now consider the corresponding long exact sequence $$\xymatrix{ \dots\ar[r] & \bigoplus\limits_{j = 1}^{c+r}H^k(\bbP^r_{K}, E^{\tensor n-1}(m+b_j)) \ar[r] & H^k(\bbP^r_{K}, E^{\tensor n}(m))  \ar[r] & }$$ $$\xymatrix{ \ar[r] & \bigoplus\limits_{i=1}^cH^{k+1}(\bbP^r_{K}, E^{\tensor n-1}(m+a_i)) \ar[r] & \dots &  }$$ Since the stability of the vector bundle $E$ implies $b_1<0$, it follows that all $a_i,b_j<0$. So the claim is true for $n=1$, because the $k$th cohomology of the sheaves $\scrO(m+a_i)$ and $\scrO(m+b_j)$ vanishes for all $k \leq r-1$. The result follows then inductively with the same long exact sequence.
\end{proof}

Recall that $\frako$ denotes the valuation ring of $\bbC_p$. We now want to consider coherent sheaves $\mathcal{E}$ on $\bbP^r_{\frako}$ defined as the kernel of the morphism $\bigoplus_{i=1}^c \scrO(a_i) \to \bigoplus_{j=1}^d\scrO(b_j)$, such that $\mathcal{E}$ is a stable kernel bundle of degree $0$ and rank $r$ on the generic fibre. Then we know that the restriction of $\mathcal{E}$ to a model $\frakX_{\frako} \subset \bbP^r_{\frako}$ of a smooth curve over $\overline{\bbQ}_p$ of sufficiently high degree is again stable on the generic fibre and that $\mathcal{E}^{\tensor n}$ does not have nontrivial global sections for all $0 < n < r$. If furthermore the reduction mod $p^q$ of this bundle on $\frakX_{\frako}$ with $q$ as in Theorem \ref{zusammen} is trivial, then the Tannaka dual group of the bundle restricted to $X_{\bbC_p}$ is connected and semisimple. We will show that all components of this group have to be of type $A$.

\begin{lemma} \label{lokfrei}
Let $\frakX$ be a reduced scheme of finite type over $\mbox{spec}(\frako_K)$, where $K \supseteq \bbQ_p$ is a finite field extension with residue field $\mathbb{F}_q$. 
A coherent sheaf $\mathcal{F}$ on $\frakX$ is locally free of rank $r$ if and only if its restriction to the generic fibre $X_K$ and to the special fibre $X_{\mathbb{F}_q}$ is locally free of rank $r$.
\end{lemma}

\begin{proof}
Since the valuation ring $\frako_K$ contains only two prime ideals, the maximal ideal $\mathfrak{m}$ and the trivial ideal $(0)$,  %(siehe \cite{bour4}, Proposition 4.6 und Proposition 4.8)
 the open embedding $X_{K} \into \frakX$ and the closed embedding $X_{\mathbb{F}_q} \into \frakX$ are a disjoint covering of the scheme $\frakX$. Due to \cite{har}, Ex. 5.8, the coherent sheaf $\mathcal{F}$ is locally free of rank $r$ if and only if the map $\phi(x) := \mbox{dim}_{k(x)}(\mathcal{F}_x \tensor k(x))$ is constant on all connected components of $\frakX$. Since the value of $\phi$ does not change after restriction to the fibres, the result follows.
\end{proof}

Next we will illustrate the goal of this section by the following example of a stable syzygy bundle on the projective plane. It was communicated to us by H. Brenner. 
 
\begin{example} \label{beispi}
Consider the syzygy sheaf $\calE := \mbox{Syz}(X^2,Y^2,pZ^2+XY)(3)$ on $\bbP^2_{\overline{\bbZ}_p}$. The generic fibre $E := \calE \tensor \overline{\bbQ}_p$ is locally free on the projective plane $\bbP^2_{\overline{\bbQ}_p}$ since the polynomials do not have common zeros, see \cite{bre}, Lemma 2.1. It is sitting in the short exact sequence $$ 0 \longto E \longto \bigoplus_{i = 1}^3 \scrO(1) \longto \scrO(3) \longto 0$$ and obviously has degree $0$ and rank $r$. Furthermore it is stable, because applying the left-exact functor $\Gamma(\bbP^2, \cdot)$ one easily sees that $E$ does not have global sections (see \cite{oko}, Lemma 1.2.5). It follows that $E$, after restriction to a smooth curve $X \subset \bbP^2_{\overline{\bbQ}_p}$ of sufficiently high degree $d$ is again stable by Theorem \ref{langer}.
Computing a lower bound for $d$ using the remark above, we see that $d > 7$ suffices.
The special fibre $\calE \tensor \overline{\bbF}_p$ is the syzygy sheaf $\mbox{Syz}(X^2,Y^2,XY)(3)$ on the projective plane $\bbP^2_{\overline{\bbF}_p}$. These monomials have no common zero on the open subset $\bbP^2_{\overline{\bbF}_p} - \{(0;0;1)\}$, so again this sheaf is locally free there and possesses the linearly independent global sections $(-Y,0,X)$ and $(0,X,-Y)$, hence it is trivial. For any model $\mathfrak{X}  \subset \bbP^2_{\overline{\bbZ}_p}$ of a smooth projective curve $X$ of degree $d > 7$, such that the special fibre $\frakX \tensor \overline{\mathbb{F}}_p$ does not contain the point $[0;0;1]$, Lemma \ref{lokfrei} implies that $\calE|_{\frakX}$ is locally free. Hence $\calE|_{\frakX}$ is a model of $E|_{X}$ with trivial special fibre $\calE|_{\frakX} \tensor \overline{\bbF}_p$. It is obvious that $\calE|_{\frakX}$ is even trivial modulo $p$. In particular, Theorem \ref{zusammen} yields that $E|_{X_{\bbC_p}}$ lies in $\mathfrak{B}^s_{X_{\bbC_p}}$ with connected Tannaka dual group $G_{E|_{X_{\bbC_p}}}$. Since $\mbox{det}(E) = \scrO$, the algebraic group $G_{E|_{X_{\bbC_p}}}$ is semisimple due to Proposition \ref{halbeinfach}. But the only faithful and irreducible two-dimensional representation of a semisimple group is the representation $G_{E|_{X_{\bbC_p}}}=\mathbf{SL}_{E_x} \subset \mathbf{GL}_{E_x}$.  
\end{example}

We want to generalize this example and prove the following theorem.

\begin{thm} \label{theo1}
Let $E$ be a stable vector bundle of degree $0$ and rank $r$ on the projective space $\bbP^r_{\bbC_p}$ having a resolution as in Theorem \ref{bohn}. Let $X \subset \bbP^r_{\overline{\bbQ}_p}$ be a smooth, connected and projective curve of degree $d>>0$. If the bundle $E$, restricted to $X_{\bbC_p}$, is trivial modulo $p^q$ for some $q$ as in Theorem \ref{zusammen}, we have:

\begin{enumerate}
\item The vector bundle $E$ restricted to the curve $X_{\bbC_p}$ is stable and lies in the category $\mathfrak{B}^s_{X_{\bbC_p}}$.
\item The Tannaka dual group $G_E \subset \mathbf{GL}_{E_x}$ of $E$ restricted to the curve $X_{\bbC_p}$ is connected and semisimple. All components of $G_E$ are of type $A$.

\item If $\dim(E_x^{\tensor r})^{G_E} = \dim\Gamma(X_{\bbC_p}, E^{\tensor r}) \leq 3$ or if $r$ is some prime-power, then $G_E$  is almost simple of type $A$. 
\end{enumerate}
\end{thm}

\begin{remark}
There are even lots of syzygy bundles that fulfill all conditions of the theorem. 
Consider a bundle $E$ sitting in the short exact sequence $$0 \longto E \longto \bigoplus_{i = 1}^{r+1}\scrO(1) \longto \scrO(r+1) \longto 0,$$ where the right morphism is defined by homogeneous polynomials $$f_i := X_0^{i-1}X_1^{r-i+1} + pg_i,\; i = 1,\dots, r+1,$$ with $g_i \in \frako[X_0,\dots,X_r]$ having no common zeros. Then for every smooth projective curve $X$ of sufficiently high degree, with a model $\frakX \subset \bbP^r_{\frako}$ such that its special fibre does not intersect the subspace $[0;0;X_2;\dots;X_r]$, one easily sees that the bundle $\calE$ modulo $p$ has $r$ linearly independent global sections and hence is trivial.
\end{remark}

\begin{proof}
The restriction of the vector bundle $E$ to $X_{\bbC_p}$ is stable and has a connected, semisimple Tannaka dual group $G_E$ due to Theorem \ref{zusammen}, Proposition \ref{halbeinfach} and the restriction theorem of Langer. Furthermore, if the degree of the curve $X_{\bbC_p}$ is high enough, we may even assume that $\Gamma(\bbP^r_{\bbC_p}, E^{\tensor n}) = \Gamma(X_{\bbC_p}, E|_{X_{\bbC_p}}^{\tensor n})$ for all $n<r$. Now we prove that $G_E$ is almost simple under the conditions of the third assertion. Denote by $\frakg_E$ the Lie algebra of $G_E$ and suppose $\frakg_E = \frakg_1 \oplus \frakg_2$ where $\frakg_1$ and $\frakg_2$ are semisimple Lie algebras. Define $V := E_x$. It follows from \cite{bour3}, page 234, Ex.18f that there are irreducible  $\frakg_i$-modules $V_i$ with $V = V_1 \tensor V_2$ as $\frakg_E$-modules. Write $r := \mbox{dim}(V)$ and $r_i := \mbox{dim}(V_i)\geq2$ and notice $V^{\tensor r} = V_1^{\tensor r_1r_2} \tensor V_2^{\tensor r_1r_2}.$ Now we can decompose $V_1^{\tensor r_1-1} = V_1^* \oplus W_1$ due to the Lemma of Schur, where $W_1$ is some semisimple $\frakg_1$-module $W_1$. Hence we have that $(V_1^* \tensor V_1)^{r_2} = \mbox{End}(V_1^{\tensor r_2})$ is a direct summand of $V_1^{\tensor r_1r_2}$. But then the $\frakg$-module $V_1^{\tensor r_2}$ is not irreducible, see again \cite{bour3}, page 234, Ex.18f. It follows that $\mbox{dim}(\mbox{End}(V_1^{\tensor r_2}))^{\frakg_E} \geq 2$. The same holds for the $\frakg$-module $V_2^{\tensor r_1}$, which shows $\dim((V^{\tensor r})^{\frakg_E}) \geq 4$. If this condition is not fulfilled, $\frakg_E$ has to be simple. Further, the Lie algebra $\frakg_E$ is also simple if $r = l^n$ for some prime $l$, because then, without loss of generality, $r_1$ divides $r_2$ and $(V^{\tensor r_2})^{\frakg_E} = (V_1^{\tensor r_2} \tensor V_2^{\tensor r_2})^{\frakg_E} \neq 0$. But since  $r_2$ is strictly smaller than $r$, this is a contradiction to Lemma \ref{nkleinerr}. In the next section we will deal with the last claim about the type of the components of the semisimple group $G_E$. 

\section{Invariants of the $n$-fold tensor product of $\frakg$-modules}
We start with describing the setting. For references concerning the representation theory of semisimple Lie algebras see for example \cite{bour3} or \cite{hum1}. Denote by $\frakg$ a finite-dimensional, semisimple Lie algebra over an algebraically closed field $K$ of characteristic $0$. Let $\mathfrak{h} \subseteq \frakg$ be a Cartan subalgebra. We then have $$\frakg =  \bigoplus_{\alpha \in \mathfrak{h}^*}\frakg_{\alpha}$$ with $\frakg_{\alpha} := \{x \in \frakg \; ; \; [hx]= \alpha(h)x \mbox{ for all } h \in \mathfrak{h}\}.$ The roots of  $\frakg$ (with respect to $\mathfrak{h}$) are the elements of the finite set $$\Phi := \{\alpha \in \mathfrak{h}^* \; ; \; \alpha \neq 0 \mbox{ and } \frakg_{\alpha} \neq 0\}.$$ Let $E_{\bbQ} \subseteq \mathfrak{h}^*$ be the vector space over $\bbQ$ generated by $\Phi$. We have  $\mbox{dim}_{\bbQ}E_{\bbQ} = \mbox{dim}_k\mathfrak{h}^*$. The vector space $E := E_{\bbQ} \tensor \bbR$ inherits a scalar product $(.\,,.)$ from $\frakg$, namely the dual of the Killing form. The set $\Phi$ is an abstract root system in the vectorspace $E$ in the sense of \cite{hum1}, Chapter III. There is a one-to-one correspondence between isomorphism-classes of semisimple Lie algebras of rank $r$ and isomorphism-classes of root systems in an $r$-dimensional $\bbR$-vector space, where the semisimple Lie algebra $\frakg = \frakg_1 \oplus \dots \oplus \frakg_n$ with simple components $\frakg_i$ corresponds to the root system $\Phi = \Phi_1 \cup \dots \cup \Phi_n$ with irreducible systems $\Phi_i$. The latter belong to the classical root systems $A_l,  l \geq 1, \; B_l,  l \geq 2, \; C_l,  l \geq 3, \; D_l, \l \geq 4$ or to the exceptional cases $E_6,E_7,E_8,F_4,G_2$. Now let $\Delta = (\alpha_1,\dots,\alpha_l)$ be a basis of the root system $\Phi$, i.e. $\Delta$ is a basis of the vector space $E$ and every $\beta \in \Phi$ is a linear combination of the $\alpha_i$ with all coefficients either nonnegative or nonpositive. If we define $\langle \lambda , \alpha\rangle \defeq \frac{2(\lambda,\alpha)}{(\alpha,\alpha)}$ for $\lambda,\alpha \in E, \alpha \neq 0$, the weight lattice of the root system $\Phi$ is the set $$\Lambda:= \{\lambda \in E \; ; \; \langle \lambda , \alpha\rangle \in \bbZ \mbox{ for all } \alpha \in \Phi\}.$$ We denote by $\Lambda_r := \bbZ\Phi$ the lattice spanned by the roots. Then $\Lambda_r \subseteq \Lambda \subseteq E$, with finite quotient $\Lambda/\Lambda_r$. This group is called the fundamental group of the root system. 
For a finite-dimensional $\frakg$-module $V$ we have $ V = \bigoplus_{\lambda \in \Lambda}V_{\lambda}$, where $V_{\lambda} := \{ v \in V \; ; \; h\cdot v = \lambda(h)v \mbox{ for all } h \in \mathfrak{h}\}.$ We call $\lambda \in \Lambda$ a weight of multiplicity $\mbox{dim}(V_{\lambda})$ of the $\frakg$-module $V$ if the vector space $V_{\lambda}$ is nontrivial. There is a one-to-one correspondence between dominant weights in $\Lambda$ and isomorphism-classes of irreducible $\frakg$-modules.\medskip

With respect to the previous section we are now interested in the following case: let $\frakg$ be a semisimple Lie algebra and $V$ an irreducible $\frakg$-module of dimension $r$, with $(V^{\tensor n})^{\frakg} = \{0\} \mbox{ for } 1\leq n < r.$ To restrict the possible Lie algebras that can occur, let us first consider the dual of $V$. There is a uniquely determined element $w_0$ in the Weyl group $W$ with the property $w_0(\Delta) = -\Delta$ (see \cite{bour2}, Chapitre VI, \S1.6, Corollaire 3). If $V = V(\lambda)$ with a dominant weight $\lambda$, we have $V^* \iso V(-w_0(\lambda))$ due to \cite{bour3}, Chapitre VIII, \S7.5, Proposition 11. The classification of simple Lie algebras shows $w_0 = -1$ for the Lie algebras of type $A_1, B_l \mbox{ for } l \geq 2, C_l \mbox{ for } l \geq 3, D_l \mbox{ for even } l, E_7,E_8,F_4,G_2.$ In all these cases $V(\lambda) \iso V(-w_0(\lambda)) \iso V(\lambda)^*$, hence $(V \tensor V)^{\frakg} = (V \tensor V^*)^{\frakg} = \mbox{End}_{\frakg}V \neq \{0\}$. There only remain the cases $A, D_l \mbox{ for odd } l, \mbox{ and } E_6$. In the sequel we will exclude the Lie algebras of type $D_l$ and $E_6$. \medskip
First, let us fix some more notations. The set of weights of the $\frakg$-module $V(\lambda)$ is denoted by $\Pi(\lambda)$, the multiplicity of a weight $\mu \in \Pi(\lambda)$ is $m_{\lambda}(\mu) := \mbox{dim}(V(\lambda)_{\mu})$.

\begin{lemma} \label{gewicht}
Let $\frakg$ be a semisimple Lie algebra and $V(\lambda)$ an irreducible $\frakg$-module with highest weight $\lambda$. Then $$m_{\lambda}(\lambda -t\alpha) = 1 \mbox{ for } 0 \leq t \leq \langle\lambda, \alpha\rangle$$ for all $\alpha \in \Delta$.
\end{lemma} 

\begin{proof}
We use the formula of Freudenthal (\cite{hum1}, Theorem 22.3) and prove the claim for an $\alpha_0 \in \Delta$. Defining $\delta := \frac{1}{2}\sum_{\alpha \succ 0}\alpha$, we have for the multiplicities $m_{\lambda}(\mu)$ the recursive formula $$((\lambda+\delta,\lambda+\delta)-(\mu+\delta,\mu+\delta))m_{\lambda}(\mu) = 2\sum_{\alpha \succ 0}\sum_{i = 1}^{\infty}m_{\lambda}(\mu+i\alpha)(\mu+i\alpha,\alpha).$$    
Now $m_{\lambda}(\lambda) = 1$ due to \cite{hum1}, Theorem 20.2. Suppose we already computed $m_{\lambda}(\lambda-(t-1)\alpha_0) = 1$. For $\mu = \lambda -t\alpha_0$ the right side of the Freudenthal formula is \begin{eqnarray*} 2\sum_{i=0}^{t-1}(\lambda-i\alpha_0,\alpha_0) & = & 2(t\lambda-\frac{t^2-t}{2}\alpha_0,\alpha_0) \\ & = & (2t\lambda-(t^2-t)\alpha_0,\alpha_0). \end{eqnarray*} Because of $\langle\delta,\alpha_0\rangle = 1$ (\cite{bour2}, Chapitre VI, \S1.10, Proposition 29) we have $2(\delta, \alpha_0) = (\alpha_0, \alpha_0)$, and computing the scalar products on the left side of the formula yields \begin{eqnarray*} & & (\lambda+\delta,  \lambda+\delta)-(\lambda-t\alpha_0+\delta, \lambda-t\alpha_0+\delta) \\ & = & (\lambda+\delta, \lambda+\delta)-((\lambda+\delta, \lambda+\delta) -2(\lambda+\delta, t\alpha_0)+(t\alpha_0, t\alpha_0)) \\ & = & 2t(\lambda+\delta,  \alpha_0) -t^2(\alpha_0, \alpha_0) \\ & = & 2t(\lambda, \alpha_0) + 2t(\delta,\alpha_0) - t^2(\alpha_0,\alpha_0) \\ & = & 2t(\lambda, \alpha_0) - (t^2-t)(\alpha_0,  \alpha_0) \\ & = & (2t\lambda-(t^2-t)\alpha_0, \alpha_0). \end{eqnarray*} 
It follows that $m_{\lambda}(\lambda-t\alpha_0) = 1$, since for $1 \leq t \leq  \langle\lambda, \alpha_0\rangle$ we have $(2t\lambda-(t^2-t)\alpha_0,  \alpha_0) \neq 0$. 
\end{proof}

\begin{lemma} \label{kompo}
Let $V(\lambda)$ be an irreducible $\frakg$-module. Then $$V(2\lambda) \oplus V(2\lambda-\alpha)\oplus \dots \oplus V(2\lambda-t\alpha) \subseteq V(\lambda) \tensor V(\lambda)$$ for $t = \langle \lambda,  \alpha\rangle$ and all $\alpha \in \Delta$.
\end{lemma}

\begin{proof}
The weights of the tensor product $V(\lambda) \tensor V(\lambda)$ are precisely the weights $\mu + \nu$ with weights $\mu,\nu \in \Pi(\lambda)$. Furthermore, after \cite{hum1}, Ex.7 on page 117 the multiplicities satisfy $$\mbox{dim}(V(\lambda) \tensor V(\lambda))_{\mu+\nu} = \sum_{\pi+\tau = \mu + \nu}\mbox{dim}V(\lambda)_{\pi}\cdot \mbox{dim}V(\lambda)_{\tau} = \sum_{\pi+\tau = \mu + \nu}m_{\lambda}(\pi)m_{\lambda}(\tau).$$ Hence we can compute the multiplicities of $2\lambda-t\alpha$ with $0 \leq t \leq \langle \lambda,  \alpha \rangle$: $$\mbox{dim}(V(\lambda)\tensor V(\lambda))_{2\lambda-t\alpha} = \sum_{i = 0}^tm_{\lambda}(\lambda-i\alpha)m_{\lambda}(\lambda-(t-i)\alpha) = t+1$$ due to Proposition \ref{gewicht}. The assertion of the lemma can be proved inductively. Since $2\lambda$ is the highest weight of $V(\lambda) \tensor V(\lambda)$ with multiplicity $1$, we conclude that $V(2\lambda)$ is a submodule of $V(\lambda) \tensor V(\lambda)$. Suppose we already showed $$V(2\lambda) \oplus V(2\lambda-\alpha)\oplus \dots \oplus V(2\lambda-(t-1)\alpha) \subseteq V(\lambda) \tensor V(\lambda).$$ Define $m:= \langle \lambda,  \alpha \rangle$ and observe that  $$\langle 2\lambda -i\alpha, \alpha \rangle = 2 \langle\lambda, \alpha \rangle - i \langle \alpha,  \alpha \rangle = 2m-2i.$$ Again we conclude using Proposition \ref{gewicht} $$m_{2\lambda-i\alpha}(2\lambda-t\alpha) = 1 \mbox{ for all } 0  \leq i \leq t.$$ Hence the $\frakg$-module $V(2\lambda)\oplus V(2\lambda-\alpha)\oplus \dots \oplus V(2\lambda-(t-1)\alpha)$ contains the weight $2\lambda-t\alpha$ with multiplicity $t$. But we showed above that $V(\lambda) \tensor V(\lambda)$ contains this weight with multiplicity $t+1$. It follows $V(2\lambda-t\alpha) \subseteq V(\lambda) \tensor V(\lambda)$.  
\end{proof}

Let us now consider a Lie algebra $\frakg$ of type $D_l$, with $l$ odd. We know from \cite{bour2}, Chapitre $VI$, Planche $IV$ that $-w_0$ is the automorphism of the root system that permutes $\alpha_{l-1}$ and $\alpha_l$ and fixes all other $\alpha_i$. The same holds for the basis $\lambda_1,\dots,\lambda_l$ of the root lattice. For a dominant weight $\lambda = \sum_{i=1}^la_i\lambda_i \in \Lambda$ with $a_{l-1} = a_l$ we hence have $V(\lambda)^* = V(-w_0(\lambda)) = V(\lambda)$, that is $(V(\lambda)\tensor V(\lambda))^{\frakg} \neq 0$. So it is enough to consider the case $a_{l-1} \neq a_l$. For symmetry reasons it is enough to check the claim for $a_{l-1} > a_l$. The Cartan matrix describes the change of basis from  $\alpha_1,\dots,\alpha_l$ to $\lambda_1,\dots,\lambda_l$, hence we get $\alpha_{l-1} = 2\lambda_{l-1}-\lambda_{l-2}$. We have seen in Lemma \ref{kompo} that the tensor product $V(\lambda) \tensor V(\lambda)$ contains the $\frakg$-modules $$V(2\lambda-t\alpha_{l-1}) = V(2a_1,2a_2,\dots,2a_{l-3},2a_{l-2}+t,2a_{l-1}-2t,2a_l)$$ for all $0\leq t \leq \langle \lambda,  \alpha_{l-1}\rangle = a_{l-1}$. In particular, for $t = a_{l-1} - a_l$ we obtain the self-dual module $V(2a_1,\dots,2a_l,2a_l)$ which yields a nontrivial $\frakg$-invariant element in $V(\lambda)^{\tensor 4}$. Then the assumption $(V(\lambda)^{\tensor n})^{\frakg} = \{0\}$ for all $1 \leq n <r$ implies $\mbox{dim}V(\lambda) \leq 4$. But the smallest irreducible $\frakg$-module with $l \geq 4$ is of dimension $2\cdot l$, see \cite{bour3}, Chapitre VII, page 214. It follows that the only possibility is $D_3$, which is isomorphic to $A_3$.\medskip
    
Next we will exclude the Lie algebra of type $E_6$. Let $\lambda = \sum_{i = 1}^6a_i\lambda_i$ denote a dominant weight. We will again show that for some sufficiently high $n$ the tensor product $V(\lambda)^{\tensor n}$ contains a self-dual submodule, but the computations are more complicated than in the previous case. Due to \cite{bour3}, Chapitre VII, Planche V, the automorphism $w_0$ acts on the weight lattice by $-w_0((a_1,a_2,a_3,a_4,a_5,a_6)) = (a_6,a_2,a_5,a_4,a_3,a_1).$ Hence we have to find a submodule in $V(\lambda)^{\tensor n}$ with highest weight satisfying $a_1 = a_6$ and $a_3 = a_5$. Because  $\lambda_2$ and $\lambda_4$ do not play any role in these computations, they are omitted in the sequel. We first assume $\lambda = (a_1,0,a_5,0)$. The Cartan matrix of $E_6$ implies the equations $\alpha_1 = 2\lambda_1- \lambda_3$ and $\alpha_5 = 2\lambda_5-\lambda_4-\lambda_6.$ One concludes using Lemma \ref{kompo} that for $2s \leq a_1$ and $2t \leq a_5$ the irreducible $\frakg$-modules with highest weight $(2a_1-2s,s,2a_5,0)$ resp. $(2a_1,0,2a_5-2t,t)$ are contained in $V(\lambda)\tensor V(\lambda)$. It follows that $V(\lambda)^{\tensor 4}$ contains the submodule of highest weight $(4a_1-2s,s,4a_5-2t,t)$. This module is self-dual if \begin{eqnarray*} 4a_1-2s &=& t \\ 4a_5-2t &=& s    \end{eqnarray*} Hence we find a self-dual submodule with \begin{eqnarray*} t &=& \frac{4}{3}(2a_1-a_5) \\ s &=& \frac{4}{3}(2a_5-a_1) \end{eqnarray*} if $t$ and $s$ are nonnegative integers. If $a_1$ and $a_5$ are divisible by $3$, we get the sufficient condition $$2a_1 \geq a_5 \geq \frac{1}{2}a_1$$for the occurence of a self-dual submodule in $V(\lambda)^{\tensor 4}$. If the parameters $a_1$ and $a_5$ do not satisfy these conditions, we have to do the following to find a submodule that does. Let $a_1 < a_5$ without restriction of generality. The tensor product $V(\lambda) \tensor V(\lambda)$ contains $V(2\lambda-a_5\alpha_5)$, with $2\lambda-a_5\alpha_5 = (2a_1,0,0,a_5);$ then $V(\lambda)^{\tensor 4}$ contains the module $V(4\lambda-2a_5\alpha_5-a_5\alpha_6)$ with highest weight $(4a_1,0,a_5,0)$. Repeating this $n$ times, one finally obtains a submodule with highest weight $(4^na_1,0,a_5,0)$ lying in the module $V(\lambda)^{\tensor 4^n}$. Now choose $n$ with  $$2 \cdot 4^na_1 \geq a_5 \geq 2 \cdot 4^{n-1}a_1= \frac{1}{2}4^na_1,$$which gives a submodule satisfying the conditions above. Hence we showed: If $V$ is an irreducible $\frakg$-module with highest weight $(a_1,0,a_5,0)$, then $V^{\tensor k}$ contains a self-dual submodule, where $$k \leq 4^n\cdot 3 \cdot 4 \leq 24 \cdot a_5.$$ It follows that $V(\lambda)^{\tensor 2k}$ has a nontrivial $\frakg$-invariant element.
It remains to start with an arbitrary $\frakg$-module $V(\lambda')$ and give an upper bound for $m$ such that $V^{\tensor m}(\lambda')$ contains a submodule of highest weight $(a_1,0,a_5,0)$. With the same arguments as above one easily sees that there is such a submodule for an $m\leq 32$, with $a_5 \leq 16\cdot\mbox{max}\{a_i'\}$. It follows that $V(\lambda')^{\tensor n} = (V(\lambda')^{\tensor 32})^{\tensor 2k}$ has a nontrivial $\frakg$-invariant element for an $n$ with $$n \leq 32\cdot 2k \leq 32\cdot 2\cdot 24 \cdot a_5 \leq 64\cdot 24 \cdot 16\mbox{max}\{a_i'\}= 24576 \cdot \mbox{max}\{a_i'\}.$$  Looking at the dimension formula of Weyl (e.g. \cite{hum1}, 24.3) $$\mbox{dim}(V(\lambda)) = \frac{\Pi_{\alpha \preceq 0}\langle\lambda + \delta, \alpha\rangle}{\Pi_{\alpha \preceq0}\langle\delta, \alpha\rangle} = c\cdot \Pi_{\alpha \preceq 0}(\sum_{i = 1}^{|\Delta|}c_i^{(\alpha)}(\lambda_i+1))$$ with a constant $c \in \bbQ$ and nonnegative integers $c_i^{(\alpha)}$ depending only on the Lie algebra $\frakg$, one finds that the map $\lambda \mapsto \mbox{dim}V(\lambda)$ grows polynomially in $\lambda_i$ and is strictly increasing, in particular it grows much faster than the linear term $24576 \cdot \mbox{max}\{a_i'\}$. With the help of a table for the dimensions of irreducible $\frakg$-modules (see e.g. \cite{mck}) one can write down all dominant weights whose associated irreducible $\frakg$-module is of dimension $\leq 24576 \cdot \mbox{max}\{a_i\}$ and not self-dual: $$\begin{array}{ccc} (a,0,0,0,0,0) & (0,0,b,0,0,0) & (1,1,0,0,0,0) \\ (1,0,1,0,0,0) & (1,0,0,0,1,0) & (2,0,0,0,0,1) \\ (0,1,1,0,0,0) & (1,2,0,0,0,0) & (2,1,0,0,0,0) \\ (3,0,0,0,0,1) \end{array}$$ with $a \leq 5$ and $b \leq 2$. We can check the claim for these cases using a computer and software like e.g. LiE, which can be found on  http://young.sp2mi.univ-poitiers.fr/~marc/LiE. Hence we also have excluded the Lie algebra $E_6$.

\begin{prop} \label{invar}
Let $\frakg$ be a semisimple Lie algebra over an algebraically closed field of characteristic $0$ and $V$ an $r$-dimensional, irreducible and faithful $\frakg$-module that satisfies $$(V^{\tensor n})^{\frakg} = \{0\} \mbox{ for }1 \leq n < r.$$ Then all components of $\frakg$ are of type $A$.
\end{prop}

\begin{proof}
If the Lie algebra $\frakg$ is simple, we showed that it has to be of type $A$. Now let  $\frakg = \oplus_{i \in I}\frakg_i$ with simple Lie algebras $\frakg_i$. It follows from  \cite{bour3}, page 234, Ex.18f that there are irreducible $\frakg_i$-modules $V_i$ of dimension  $r_i$ with $V = \tensor_{i\in I}V_i$. If the Lie algebra $\frakg_1$ say was not of type $A$, there would be some $s < r_1$ with $V_1^{\tensor s}$ possessing a nontrivial $\frakg_i$-invariant element. But then the module $(\tensor_{i \in I}V_i)^{\tensor t}$ would also have an $\frakg_i$-invariant element for $t = s\prod_{i \in I\backslash\{1\}}r_i<r$, which is a contradiction. 
\end{proof}  

This finishes the proof of Theorem \ref{theo1}.
\end{proof}

\begin{remark}
One might ask how to improve the upper bounds for the occurence of a nontrivial $\frakg$-invariant element in $V^{\tensor n}$ for an irreducible module $V$ of a semisimple Lie algebra $\frakg$. Obviously there is the following lower bound: if $\mu \preceq \lambda$ are dominant weights such that $n\lambda-\mu$ does not lie in the lattice of roots $\Lambda_r$, then $\Pi(V(\lambda)^{\tensor n})$ can not contain the weight $\mu$ since all these weights have the form $n\lambda-\sum_{\alpha \in \Delta}k_{\alpha}\alpha$ with integral $k_{\alpha}$. Hence the module $V(\mu)$ can not be contained in $V(\lambda)^{\tensor n}$ and the maximal order of all elements of the fundamental group $\pi:= \Lambda/\Lambda_r$ is a lower bound. We propose the following conjecture:   
\end{remark}

\begin{conj} \label{verm}
Let $\frakg$ be a semisimple Lie algebra and $V(\lambda)$ an irreducible $\frakg$-module of dimension $r$, with highest weight $\lambda$. Let $\mu \preceq \lambda$ be some dominant weight in $\Pi(\lambda)$. Then $V(\lambda)^{\tensor n}$ contains the irreducible $\frakg$-module $V(\mu)$ for an $n \leq \mbox{max}_{g \in \pi}\mbox{ord}(g)$ if $\pi$ is nontrivial and for $n = 2$ otherwise.
\end{conj}    

In particular, we could strenghten Proposition \ref{invar} by

\begin{cor} \label{korverm}
Assume that Conjecture \ref{verm} holds. Under the condition of Proposition \ref{invar} the Lie algebra $\frakg$ is simple of type $A_{r-1}$ and $V$ is the standard $r$-dimensional representation $V(\lambda_1)$ or its dual $V(\lambda_{r-1})$.
\end{cor}

\begin{proof}
The fundamental group $\pi$ of a Lie algebra of type $A_l$ is cyclic of order $l+1$, see for example \cite{hum1}, page 68. Hence $A_l$ with an $l<{r-1}$ is impossible, because its representations do not satisfy the condition $(V^{\tensor n})^{\frakg} = \{0\} \mbox{ for } 1\leq n < r$. For $l \geq r$ there are no nontrivial modules of dimension $r$. It follows that $\frakg$ has to be of type $A_{r-1}$, and the only possible modules of dimension $r$ are those mentioned above.
\end{proof}

\section{Analytic monodromy groups of vector bundles on p-adic curves}

This section deals with some properties of the image of a continuous representation $\rho_E: \pi_1(X,x) \to \mbox{GL}(\bbC_p)$ attached to a vector bundle $E$ with potentially strongly semistable reduction on the curve $X_{\bbC_p}$. We call this image the analytic monodromy group $G_{\rho_E}$ of the bundle $E$. We will investigate under which conditions this group is a p-adic analytic group and show how its Lie algebra is connected to the Lie algebra of the algebraic monodromy group. For stable bundles we will give a criterion under which conditions $G_{\rho_E}$ is defined over some finite field extension of $\bbQ_p$. Recall the following property of pro-finite groups.

\begin{defn}
A topological group $G$ is topologically generated by elements $g_1,\dots,g_r \in G$, if the closure of the subgroup generated by these elements is all of $G$.
\end{defn}

\begin{prop} \label{erzeugt}
Let $G$ be a pro-finite group, $H$ a closed subgroup and $d \in \bbN$ a positive integer. If for all open normal subgroups $N$ of $G$ the finite quotient $HN/N$ is generated by $d$ elements, then $H$ is topologically generated by $d$ elements.
\end{prop}
\begin{proof}
\cite{dix}, Proposition 1.5ii.
\end{proof}

If we fix an abstract field embedding $\overline{\bbQ}_p \subseteq \bbC$, Grothendieck's comparison theorem (\cite{sga1}, expos\'e XII, Corollaire 5.2) shows that the algebraic fundamental group of a smooth projective curve $X$ over $\overline{\bbQ}_p$ is the pro-finite completion of the topological fundamental group. Since the latter is the free group of $2g$ generators $u_i,v_i, \, i=1,\dots,g$ modulo the relation $\Pi_{i = 1}^gu_iv_iu_i^{-1}v_i^{-1} = 1$, where $g$ denotes the genus of the curve $X$, we find that each finite quotient group is also generated by $2g$ elements. Hence the algebraic fundamental group $\pi_1(X,x)$ is topologically generated by $2g$ elements because it has the same finite quotients. \medskip

Furthermore, it was shown in \cite{dw1} that the image of a continuous representation  $\pi_1(X,x) \to \mbox{GL}_r(\bbC_p)$ can always be considered as lying in $\mbox{GL}_r(\frako)$: For an additive category $\mathcal{C}$ we denote by $\mathcal{C} \tensor \bbQ$ the category having the same objects as $\mathcal{C}$, with morphisms $\mbox{Hom}_{\mathcal{C} \tensor \bbQ}(A,B) = \mbox{Hom}_{\mathcal{C}}(A,B)\tensor \bbQ$. Then the natural functor $\mathbf{Rep}_{\pi_1(X,x)}(\frako) \tensor \bbQ \to \mathbf{Rep}_{\pi_1(X,x)}(\bbC_p)$ is essentially surjective, i.e. for every $\pi_1(X,x)$-module $V$ over $\bbC_p$ there is a $\pi_1(X,x)$-module $\Gamma$ over $\frako$ and an isomorphism $V \iso \Gamma \tensor \bbQ $, see the proof of Proposition 22 in \cite{dw1}. \medskip

Recall that for $n \in \bbN$ we denote $\frako_n := \frako/p^n\frako$. For a compact topological group $G \subseteq \mbox{GL}_r(\frako)$ let $G_n$ be the open normal subgroup $G_n := G \cap \mbox{ker}(\mbox{GL}_r(\frako) \longto \mbox{GL}_r(\frako_n)).$ We have $\bigcap_{n \in \bbN}G_n = \{1\}$, hence the normal subgroups $G_n$ establish a base for the neighbourhood of the unit element. Since $\mbox{GL}_r(\frako)$ carries a Hausdorff topology the group $G$ is a pro-finite group with $G = \projlim_{n \in \bbN}G/G_n.$

\begin{prop} \label{struktur}
Let $G \subset \mbox{GL}_r(\frako)$ be a compact subgroup and $n_0 \in \bbN$. For all $n \geq n_0$ there is a commutative diagramm with exact rows  $$\xymatrix{1 \ar[r] & G_n/G_{n+n_0} \ar[r] \ar@{^{(}->}[d] & G/G_{n+n_0} \ar[r] \ar@{^{(}->}[d]  & G/G_n \ar[r] \ar@{^{(}->}[d] & 1 \\ 1 \ar[r] & M_r(p^n\frako_{n+n_0}) \ar[r]^-{A \mapsto 1_r+A} & \mbox{GL}_r(\frako_{n+n_0}) \ar[r] & \mbox{GL}_r(\frako_n) \ar[r] & 1.}$$ In particular, $G_n/G_{n+n_0}$ is a finite abelian $p$-group isomorphic to $(\bbZ/p^{n_0})^{k_n}$ for some $k_n \geq 0$ and $G_1 \subseteq G$ is an open pro-$p$ subgroup.
\end{prop} 

\begin{proof}
The injectivity of the vertical morphisms in the middle and on the right is obvious, the morphism on the left is just the restriction of the middle one. Since $G_n/G_{n+n_0}$ is a finite group and $M_r(p^n\frako_{n+n_0})$ is abelian of exponent $p^{n_0}$, the group $G_n/G_{n+n_0}$ is also an abelian $p$-group of exponent $p^{n_0}$, so $G_n/G_{n+n_0}\iso(\bbZ/p^{n_0})^k$ for some $k \geq 0$. In particular, the groups $G_1/G_n$ are finite $p$-groups for all $n \in \bbN$. The group $G_1$ is pro-finite as an open subgroup of $G$, and because $G_1 = \projlim_{n \in \bbN}G_1/G_{n+1}$ it is even an open pro-$p$ subgroup.
\end{proof}

Since the fundamental group $\pi_1(X,x)$ is in particular a compact topological group, this is also true for the image $G_{\rho}$ of a continuous representation $\rho: \pi_1(X,x) \to \mbox{GL}_r(\bbC_p)$. Due to Proposition \ref{struktur} the group $G_{\rho}$ contains an open pro-$p$ group. A very important class of pro-finite groups are $p$-adic analytic groups. We will briefly sketch some of their properties. A good reference for a group-theoretical approach to this theory is \cite{dix}. A topological group $G$ is called a $p$-adic analytic group if it carries the structure of a $p$-adic analytic manifold such that the multiplication $\cdot:G \times G \to G, \;(x,y) \mapsto xy$ and the inverse $i: G \to G, \;x \mapsto x^{-1}$ are analytic maps. The category of $p$-adic groups is closed under closed subgroups, quotients and extensions. There is the following purely group-theoretical characterization of $p$-adic analytic groups by Lazard (see \cite{laz} or \cite{dix}): a pro-$p$ group $G$ is called powerful if $G/\overline{G^p}$ is abelian for odd primes $p$ resp. if $G/\overline{G^4}$ is abelian for $p = 2$. Further, $G$ has finite rank $\mbox{rk(G)}$ if $\mbox{rk}(G) := \mbox{sup}\{\mbox{rk}(G/N) \; ; \; N \subseteq G \mbox{ open normal subgroup }\}$ is finite. The rank of an abstract finite group is defined as $\mbox{sup}\{d(H) \; ; \; \mbox{ subgroups } H \subseteq G\}$, where $d(H)$ denotes the cardinality of some minimal set of generators of $H$. 

\begin{thm}[Lazard] \label{lazard}
Let $G$ be a topological group. The following statements are equivalent:
\begin{enumerate}
\item
The group $G$ is $p$-adic analytic.
\item
The group $G$ contains a pro-$p$ subgroup of finite rank.
\item
The group $G$ contains an open, finitely generated, powerful pro-$p$ subgroup.
\end{enumerate}
\end{thm}

We will use this theorem to give a criterion for a vector bundle with potentially strongly semistable reduction to correspond to a continuous representation $\rho_E$ with $p$-adic analytic image.

\begin{prop} \label{dicht}
Let $G \subset \mbox{GL}_r(\frako)$ be a compact and topologically finitely generated group. Using the notations of Proposition \ref{struktur}, if $k_n = k$ for all $n > n_0$ and some $k \in \bbN$, then the group $G_n^{p^{n_0}}$ is dense in $G_{n+n_0}$ for all $n > n_0$.
\end{prop}

\begin{proof} 
We have $G_n/G_{n+n_0} \subseteq M_r(p^n\frako_{n+n_0})$ for all $n > n_0$, i.e. there is the commutative diagram 
 $$\xymatrix{G_n/G_{n+n_0} \ar[r]^{g^{p^{n_0}}} \ar@{^{(}->}[d] & G_{n+n_0}/G_{n+2n_0} \ar@{^{(}->}[d] \\ M_r(p^n\frako_{n+n_0}) \ar@{^{(}->}[r]^{p^{n_0}\cdot g} & M_r(p^{n+n_0}\frako_{n+2n_0}).}$$ Hence the upper vertical morphism is injective and it even has to be an isomorphism because of $G_n/G_{n+n_0} \iso (\bbZ/p^{n_0})^k \iso G_{n+n_0}/G_{n+2n_0}$.
It follows that $G_{n+n_0} = (G_n)^{p^{n_0}}G_{n+2n_0}$ and for the same reason $G_{n+2n_0} = (G_{n+n_0})^{p^{n_0}}G_{n+3n_0}$. We then have the equality $G_{n+n_0} = (G_n)^{p^{n_0}}G_{n+3n_0}$, and if one continues this way one gets $G_{n+n_0} = (G_n)^{p^{n_0}}G_{n+ln_0} \mbox{ for all } l \in \bbN,$ with $\bigcap_{l \in \bbN} G_{n+ln_0} = \{1\}.$ For every subset $X$ of a pro-finite group $G$ we have $\overline{X} = \bigcap_{N \subseteq G}XN$, where $N$ runs over all open normal subgroups of $G$, see \cite{dix}, Proposition 1.2(iii). It is easy to see that it suffices to take an arbitrary subset of open normal subgroups with trivial intersection. We find $\overline{(G_{n})^{p^{n_0}}} = \bigcap_{l \in \bbN}(G_{n})^{p^{n_0}}G_{n+ln_0} = G_{n+n_0}$ for all $n > n_0$.
\end{proof}

\begin{cor} \label{equi}
For a compact group $G \subset \mbox{GL}_r(\frako)$ and $k \in \bbN$ the following statements are equivalent:
\begin{enumerate}
\item
The group $G$ is $p$-adic analytic of dimension $k$.
\item
$\#G_n \sim cp^{kn}$ with some $c \in \bbN$.
\item
In Proposition \ref{struktur} the equality $k_n = k$ holds for almost all $n \in \bbN$.
\end{enumerate}
\end{cor}
  
\begin{proof}
First, let $G$ be $p$-adic analytic. Then there is some $n_0 \in \bbN$ such that the open pro-$p$ subgroups $G_n$ for $n > n_0$ are of finite rank $k$ due to Theorem \ref{lazard}, this means $$\mbox{sup}\{\mbox{rk}(G_n/N) \; ; \; N \subseteq G_n \mbox{ open normal subgroups }\} = k.$$ In particular, the quotient $G_n/G_{n+1} = (\bbZ/p)^{k_n}$ is of rank $\leq k$ and hence $k_n \leq k$, with equality for almost all $n$. This shows $1 \Rightarrow 2.$

The equivalence $2 \Leftrightarrow 3$ is obvious because of the short exact sequence $$0 \longto (\bbZ/p)^{k_n} \longto G/G_{n+1} \longto G/G_n \longto 1$$ and the inequality $k_{n+1} \geq k_n$.

The implication $3 \Rightarrow 1$ follows from Proposition \ref{dicht}. The group $G_{n_0}$ with $k_{n_0} = k$ satisfies the equation $\overline{G_{n_0}^{p^2}} = G_{n_0+2}$. Hence for $n_0 \geq 2$ the quotient $G_{n_0}/\overline{G_{n_0}^{p^2}} = G_{n_0}/G_{n_0+2} = (\bbZ/p^2)^k$ is an abelian group and $G_{n_0}$ an open, powerful and topologically finitely generated pro-$p$ subgroup. The group $G$ is $p$-adic analytic due to Theorem \ref{lazard} and has rank $k$.
\end{proof}

Corollary \ref{equi} gives a sufficient criterion for $G_{\rho_E}$ being a $p$-adic analytic group where one only has to count the degree of certain trivializing coverings of $X$.

\begin{cor} \label{Laza}
 Let $E$ be a vector bundle with potentially strongly semistable reduction of degree $0$ on the curve $X_{\bbC_p}$.
Suppose there is a model $\frakX$ of $X$ and a vector bundle $\calE$ on $\frakX_{\frako}$ such that for all $n \in \bbN$ there exists a morphism $\pi: \mathcal{Y} \to \frakX$ in the category $\mathcal{S}^{good}_{\frakX, D}$ with the properties \begin{enumerate} \item $\pi^*_n(\calE_n)$ is the trivial bundle \item The morphism $\pi_{\overline{\bbQ}_p}: Y := \mathcal{Y} \tensor \overline{\bbQ}_p \to X$ of the generic fibres is of degree $\leq cp^{kn}$ with constants $c, k \in \bbN$ that are independent from $n$. \end{enumerate} 
Then the group $G_{\rho_E}$ is a $p$-adic analytic group. 
\end{cor}

\begin{proof}
Due to the construction of the representations $\rho_E$ and $\rho_{\calE}$ in \cite{dw2} we have $G_{\rho_E} \iso G_{\rho_{\calE}} \subset \mbox{GL}(\calE_{x_{\frako}})$ and the subgroup $G_n := G_{\rho_{\calE, n}}$ is a quotient of $\mbox{Aut}_{X\backslash D}(Y\backslash \pi_{\overline{\bbQ}_p}^*D).$ Since $Y\backslash \pi_{\overline{\bbQ}_p}^*D$ is connected because of the definition of $\mathcal{S}^{good}_{\frakX, D}$, we have $$\mbox{Aut}_{X\backslash D}(Y \backslash \pi_{\overline{\bbQ}_p}^*D) \subseteq \mbox{Hom}_{X\backslash D}(x,Y\backslash \pi_{\overline{\bbQ}_p}^*D).$$ Then $\#G_n \leq cp^{kn}$ for all $n \in \bbN$, and the corollary above implies that $G_{\rho_E}$ is a $p$-adic analytic group.
\end{proof}
  
Despite the fact that the conditions above are hard to check, one knows lots of nontrivial examples where the image $G_{\rho}$ of a continuous morphism $\rho: G \to \mbox{GL}_r(\bbC_p)$, with $G$ a compact group, is $p$-adic analytic. First, if $G$ is topologically finitely generated and solvable, it is easy to see that it has to be a $p$-adic analytic group. Namely, there is a sequence of closed normal subgroups $G_{\rho} = H_n \supseteq H_{n-1} \supseteq \dots \supseteq H_0 = \{1\}$ with abelian quotients $H_i/H_{i-1}$. The subgroups $H_i$ contain an open and topologically finitely generated pro-$p$ subgroup, and so do the quotient groups $H_i/H_{i-1}$, see \cite{dix}, Proposition 1.11. Since the latter are also abelian, they are in particular powerful and hence $p$-adic analytic. Then $G_{\rho}$ is $p$-adic analytic as successive extension of such groups. In particular, if a vector bundle $E$ on the curve $X_{\bbC_p}$ is a successive extension of line bundles of degree $0$, it has potentially strongly semistable reduction, see Theorem \ref{hauptthm}, and the associated group $G_{\rho_E}$ is solvable, hence $p$-adic analytic. \smallskip

Second, $G_{\rho}$ is $p$-adic analytic if it lies in $\mbox{GL}_r(\overline{\bbQ}_p)$, because in this case it is already defined over a finite field extension $K \supseteq \bbQ_p$, see for example \cite{kat2}, Remark 9.0.7, and $\mbox{GL}_r(K)$ is a p-adic analytic group. In the sequel we will show that for certain polystable vector bundles $E$ as above this is in fact also necessary for $G_{\rho_E}$ to be a $p$-adic analytic group. To achieve this goal, we have to take a closer look at the Lie algebra of $G_{\rho_E}$. \medskip
     
A pro-$p$ group $U$ of finite rank is called uniform if every powerful open subgroup $H \subseteq U$ satisfies $d(H) = d(U)$. From now on $U$ always denotes a uniform pro-$p$ group. We will briefly describe how to attach a Lie algebra to $U$, see for example \cite{dix}. Consider the $\bbQ_p$-algebra $A:= \bbQ_p[U]$. It carries a norm which is compatible with the topologies of $\bbQ_p$ and $U$, i.e. $||\lambda a|| \leq |\lambda|\cdot ||a||$ for $\lambda \in \bbQ_p$ and $a \in A$ and $||g-1||\leq p^{-n}$ for $g \in U_n, \;n \geq 1$ (see \cite{dix}, page 158 and Definition 8.12). Let $\hat{A}$ be its completion with respect to this norm. Furthermore, we define \begin{eqnarray*} \hat{A}_0 = \{x \in \hat{A} \; ; \; ||x||\leq p^{-1}\} \mbox{ for } p \neq 2 \\ \hat{A}_0 = \{x \in \hat{A} \; ; \; ||x|| \leq 2^{-2}\} \mbox{ for } p = 2\end{eqnarray*} The series $\mbox{log}(1+x):=\sum_{n = 1}^{\infty}(-1)^{n+1}\frac{x^n}{n}$ and $\mbox{exp}(x):=\sum_{n = 0}^{\infty}\frac{x^n}{n!}$ are defined on all of $\hat{A}_0$ and are inverse to each other. The natural Lie bracket $[x,y] := xy-yx$ gives a Lie algebra structure on $\hat{A}$, and one has $(U-1) \subseteq \hat{A}_0$ for $ p>2$ resp. $(U_2-1) \subseteq \hat{A}_0$ for $p=2$. Define $\Lambda := \mbox{log}(U)$ for $p \neq 2$ resp. $\Lambda := \mbox{log}(U^2)$ for $p =2$, which are $\bbZ_p$-Lie subalgebras of $\hat{A}$, see \cite{dix}, Corollary 8.15. \smallskip 

Now let $G$ be a $p$-adic analytic group. It contains a uniform pro-$p$ group $U$ (see \cite{dix}, Theorem 3.13), and one defines the Lie algebra $\frakg := \bbQ_p \tensor_{\bbZ_p}\Lambda(U).$ One can show that this definition is independent from a choice of the uniform subgroup $U$. If $f:G \to G'$ is a morphism of $p$-adic analytic groups and $U' \subseteq G'$ an open, uniform subgroup, one finds an open, uniform subgroup $U \subseteq f^{-1}(U')$. The restriction of the morphism $f|_{U}: U \to U'$ induces a morphism of Lie algebras $f^*: \hat{A}_0(U) \to \hat{A}_0(U')$ that restricts to a morphism $\Lambda(U) \to \Lambda(U')$. Finally, the $\bbQ_p$-linear continuation $f^*:\frakg(U) \to \frakg(U')$ gives a full, essentially surjective functor from the category of $p$-adic analytic groups to the category of finite-dimensional $\bbQ_p$-Lie algebras. For $U_0 := \mbox{ker}(\mbox{GL}_r(\frako) \to \mbox{GL}_r(\frako/p^{\frac{1}{p-1}}))$ there is the following lemma.
   
\begin{lemma} \label{lie}
Let $G \into U_0 \subset\mbox{GL}_r(\frako)$ be a continuous and faithful representation. The  $p$-adic logarithm $\mbox{log}_p: U_0 \into M_r(\bbC_p)$ induces an isomorphism of $\bbZ_p$-Lie algebras $\Lambda \To{\sim} \mbox{log}_p(G).$
\end{lemma}

\begin{proof}
Consider the diagram $$\xymatrix{G \subset \hat{A}_0 \ar@<0.7ex>[r]^-{log} \ar@<-3ex>[d]_{log_p} & \hat{A}_0 \ar@<1ex>[l]^-{exp} \\ M_r(\bbC_p). & }$$ Since the $p$-adic logarithm $\mbox{log}_p$ on $U_0$ coincides with $\mbox{log}$ on $\hat{A}_0$ and since the Lie bracket on $M_r(\bbC_p)$ coincides with that on $\hat{A}_0$, one sees at once that $$\Lambda \Longto{\mbox{exp}} G \Longto{\mbox{log}_p} M_r(\bbC_p)$$ gives an isomorphism of Lie algebras $\Lambda \To{\sim} \mbox{log}_pG \subseteq M_r(\bbC_p)$.
\end{proof}

Now let $\rho: G \to \mbox{GL}(V)$ be as above a continuous, faithful representation of a $p$-adic analytic group $G$ and let $V$ be a finite-dimensional vector space. We will describe how the Lie algebras of $G$ and its Zariski closure $\overline{G}$ are related. For this, let $\frakg$ and $\overline{\frakg}$ be the Lie algebras of $G$ and $\overline{G}$. The Lie algebra $\frakg$ contains a maximal solvable ideal $r(\frakg)$, the radical of $\frakg$, and the quotient $\frakg^{ss} := \frakg/r(\frakg)$ is a semisimple Lie algebra. With Lemma \ref{lie} we obtain a faithful representation $\frakg \into \mbox{End}_{\bbC_p}(V) \iso M_r(\bbC_p)$. We denote by $\frakg({\bbC_p})$ the $\bbC_p$-linear continuation of $\frakg$ in $\mbox{End}_{\bbC_p}(V)$ and show that one gets embeddings  $r(\frakg) \into r(\frakg({\bbC_p}))$ and $\frakg^{ss} \into\frakg({\bbC_p})^{ss}$. \smallskip

 Consider the short exact sequence of Lie algebras that arises from tensoring with $\bbC_p$: $0 \to r(\frakg \tensor \bbC_p) \to \frakg \tensor \bbC_p \to (\frakg \tensor \bbC_p)^{ss} \to 0.$ It follows from \cite{bour1}, Chapitre I, \S5.6 that $r(\frakg \tensor \bbC_p) = r(\frakg) \tensor \bbC_p$ and hence $(\frakg \tensor \bbC_p)^{ss} = \frakg^{ss} \tensor \bbC_p$. The image of $r(\frakg)\tensor \bbC_p$ in $\frakg({\bbC_p})$ is a solvable subalgebra and even an ideal because of the surjection $\frakg \tensor \bbC_p \onto \frakg({\bbC_p})$. Hence the radical of $\frakg \tensor \bbC_p$ is mapped onto the radical of $\frakg({\bbC_p})$ and we find $\frakg \cap r(\frakg({\bbC_p})) \supseteq r(\frakg)$. Since this intersection is a solvable ideal in $\frakg$, this yields \begin{eqnarray} \frakg \cap r(\frakg({\bbC_p})) = r(\frakg). \label{radikal} \end{eqnarray} Hence $r(\frakg)\tensor \bbC_p \to r(\frakg({\bbC_p}))$ is surjective  and we obtain the commutative diagram with exact rows \begin{equation} \xymatrix{0 \ar[r] & r(\frakg) \tensor \bbC_p \ar[r] \ar@{>>}[d] & \frakg \tensor \bbC_p \ar[r] \ar@{>>}[d] & \frakg^{ss} \tensor \bbC_p \ar[r] \ar@{>>}[d] & 0 \\ 0 \ar[r] & r(\frakg({\bbC_p})) \ar[r] & \frakg({\bbC_p}) \ar[r] & \frakg({\bbC_p})^{ss} \ar[r] & 0 .}\label{dia1}\end{equation} The surjectivity of the right vertical morphism follows at once from the surjectivity of the middle one. There also is the commutative diagram with exact rows and injective vertical morphisms \begin{equation} \xymatrix{0 \ar[r] & r(\frakg) \ar[r] \ar@{^{(}->}[d] & \frakg \ar[r] \ar@{^{(}->}[d] & \frakg^{ss}  \ar[r] \ar@{^{(}->}[d] & 0 \\ 0 \ar[r] & r(\frakg({\bbC_p})) \ar[r] & \frakg({\bbC_p}) \ar[r] & \frakg({\bbC_p})^{ss} \ar[r] & 0 .} \label{dia2} \end{equation} The injectivity on the right follows again from  (\ref{radikal}).

\begin{cor} \label{genaudann}
The Lie algebra $\frakg$ is solvable resp. semisimple if and only if $\frakg({\bbC_p})$ is solvable resp. semisimple.
\end{cor}

\begin{proof}
Both assertions follow at once from the commutative diagrams above.
\end{proof}

For a subset $M \subseteq \mbox{End}_{\bbC_p}(V)$ we denote by $a(M)$ the smallest algebraic Lie algebra containing $M$. Obviously the Lie algebra of the algebraic group $\overline{G}$ equals $a(\frakg)$ with inclusions $\frakg \subseteq \frakg({\bbC_p}) \subseteq a(\frakg) = \overline{\frakg}$.

\begin{prop} \label{algebr}
Let $\rho$ as above be a semisimple representation of rank $\geq 2$ with trivial determinant. Then $\frakg \subseteq \frakg({\bbC_p}) = a(\frakg) = \overline{\frakg}.$
\end{prop}

\begin{proof}
The condition implies that the $\frakg$-module $V$ is semisimple with vanishing trace, hence this also holds for $V$ considered as $\frakg({\bbC_p})$-module. So the Lie algebra  $\frakg({\bbC_p})$ has to be semisimple as well. Due to \cite{bor}, Corollary 7.9, the derived Lie algebra $[\frakg({\bbC_p}),\frakg({\bbC_p})]$ is algebraic. But since $\frakg({\bbC_p})$ is semisimple, we have $[\frakg({\bbC_p}),\frakg({\bbC_p})] = \frakg({\bbC_p})$. Hence $a(\frakg) = \frakg({\bbC_p})$.
\end{proof}

Applying this for polystable vector bundles $E$ with potentially strongly semistable reduction, assuming that the image $G_{\rho_E}$ is $p$-adic analytic, and denoting the Lie algebra of $G_{\rho_E}$ with $\frakg$ and the Lie algebra of the algebraic monodromy  group $G_E$ with $\frakg_E$, we obtain the result:

\begin{cor} \label{corollary10}
Let $E$ be a polystable vector bundle with potentially strongly semistable reduction of degree $0$ and rank $\geq 2$, with trivial determinant bundle and such that the associated representation $\rho_E$ is semisimple and has a $p$-adic analytic image. Then  $\frakg \subseteq \frakg({\bbC_p}) = a(\frakg) = \frakg_E.$ Furthermore, the Lie algebra $\frakg$ is semisimple and $\frakg_E$ is a direct summand of $\frakg_{\bbC_p} := \frakg \tensor \bbC_p$. 
\end{cor}

\begin{proof}
It follows from Proposition \ref{gleich} and \ref{halbeinfach} that $G_E$ is a semisimple algebraic group and that its Lie algebra is the smallest algebraic Lie algebra containing $\frakg$. The first assertion can be derived from Proposition \ref{algebr}, and due to Corollary \ref{genaudann} the Lie algebra $\frakg$ is semisimple, and so is $\frakg_{\bbC_p}$ with a surjection $\frakg_{\bbC_p} \onto \frakg({\bbC_p})$.    
\end{proof}

We want to answer the question under which conditions the image $G_{\rho_E}$ is algebraic, i.e. such that there exists a change of basis of $E_x$ with $G_{\rho_E} \subset \mbox{GL}_r(\overline{\bbQ}_p)$. Recall that $G_{\rho_E}$ is a $p$-adic analytic group in this case. Consider a  semisimple $\bbQ_p$-Lie algebra $\frakg$ and fix a Cartan subalgebra $\mathfrak{h}$ of $\frakg$. There is a finite field extension $L \supseteq \bbQ_p$ such that $\frakg_L$ is split, that is the endomorphisms $\mbox{ad}_{\frakg_L}(h)$ are simultaneously diagonalizable  for all  $h \in \mathfrak{h}_L$. If the $\bbQ_p$-Lie algebra $\frakg$ is embedded into $\mbox{End}_{\bbC_p}(V)$, we again denote by $\frakg(L)$ for $L \supseteq \bbQ_p$ the $L$-Lie algebra in $\mbox{End}_{\bbC_p}(V)$ generated by $\frakg$, which again is a split semisimple Lie algebra. This endowes $V$ with an action of $\frakg(L)$, and the associated $\frakg(L) \tensor_L \bbC_p$-module $V$ is defined by a highest dominant weight $\lambda$. But $\frakg(L)$ and $\frakg(L) \tensor_L \bbC_p$ have the same root system, and the weight $\lambda$ also corresponds to a $\frakg(L)$-action on some $L$-vector space $W$. In particular, there is a $\frakg(L)$-equivariant isomorphism $V \iso W \tensor_L \bbC_p$. Hence there is a change of basis of $V$ such that the representation $\frakg \into \mbox{End}_{\bbC_p}(V)$ is already defined over $L$. This gives the following result, where we call the group $G \subset \mbox{GL}_r(\bbC_p)$ virtually algebraic if it  contains an open subgroup which lies after a change of basis in $\mbox{GL}_r(\overline{\bbQ}_p)$.

\begin{thm} 
Let $V$ be a $\bbC_p$-vector space and let $\rho: G \to \mbox{GL}(V)$ be a continuous, semisimple and faithful representation of a compact and topologically finitely generated group $G$, with $\dim(V) \geq 2$. Furthermore, let $\mbox{det}^n(\rho)$ be trivial for some $n \in \bbN$. The group $G$ is virtually algebraic if and only if $G$ is a $p$-adic analytic group.
\end{thm} 

\begin{proof}
We already mentioned above that the group $G$ is $p$-adic analytic if it contains an open subgroup $H$ defined over  $\overline{\bbQ}_p$. To show the converse, first recall that the Lie algebra $\frakg({\bbC_p}) \subseteq \mbox{End}_{\bbC_p}(V)$ is semisimple under the conditions of the theorem. By Corollary \ref{genaudann} and the considerations above we then find a finite field extension $L \supseteq \bbQ_p$  such that $\frakg(L)$ is a split semisimple Lie algebra, hence the embedding $\frakg(L) \subseteq \mbox{End}_{\bbC_p}(V)$ is already defined over the field $L$. The $p$-adic logarithm is a continuous and bijective map from an sufficiently small open, uniform subgroup $H$ of $G$ onto the $\bbZ_p$-Lie algebra $\Lambda(H)$, with $\frakg = \Lambda(H) \tensor \bbQ_p$, see Lemma \ref{lie}. We find that $H \subseteq \mbox{GL}(V)$ is also defined over $L$, since the valuation on $L$ is complete and hence $\mbox{exp}(\Lambda(H)) \subseteq \mbox{GL}_r(L)$.
\end{proof}

\begin{cor} \label{analmono}
Under the conditions of Corollary \ref{corollary10}, $G_{\rho_E}$ is virtually algebraic if and only if $G_{\rho_E}$ is $p$-adic analytic.
\end{cor}

Here  $G_{\rho_E}$ being virtually algebraic is equivalent to the existence of a finite \'etale covering $f:Y \to X$ such that $G_{\rho_{f^*E}}$ is defined over $\overline{\bbQ}_p$.
If we again assume that $G_{\rho_E}$ is a $p$-adic analytic group, we will show, using Theorem \ref{galois}, that its Lie algebra carries a continuous representation of the absolute Galois-group $\mbox{Gal}_K$ and explain some cases where one can describe this action in more detail.

\begin{prop} \label{aaa}
Let $E$ be a vector bundle with potentially strongly semistable reduction of degree $0$. Assume that the associated representation $\rho_E: \pi_1(X,x) \to \mbox{GL}(E_x)$ has a $p$-adic analytic image. For all $\sigma \in \mbox{Gal}_{\bbQ_p}$ the continuous isomorphism $\mathbf{C}_{\sigma}: \mbox{GL}(E_x) \Longto{\sim} \mbox{GL}({^{\sigma}\!E_x})$ induces an isomorphism $$\sigma_*:G_{\rho_E} \Longto{\sim}G_{\rho_{^{\sigma}\!\!E}}$$ of $p$-adic analytic groups. If $X$, $x \in X$ and $E$ are already defined over $K$, one in particular obtains for all $\sigma \in \mbox{Gal}_K$ an automorphism $$\sigma_*:G_{\rho_E} \Longto{\sim} G_{\rho_E}$$ of $p$-adic analytic groups.
\end{prop}

\begin{proof}
By Theorem \ref{galois}  we get a commutative diagram $$\xymatrix{ \pi_1(X,x) \ar[r]^{\rho_E} \ar[d]_{\wr}^{\sigma_*} & \mbox{GL}(E_x) \ar[d]_{\wr}^{\sigma_*} \\ \pi_1({^{\sigma}\!X},{^{\sigma}}\!x) \ar[r]^{\rho_{^{\sigma}\!\!E}} & \mbox{GL}({^{\sigma}\!E_x}).}$$ One easily sees that the continuous morphism $\sigma_*: \mbox{GL}(E_x) \To{\sim} \mbox{GL}(^{\sigma}\!E_x)$ restricts to a continuous isomorphism $\sigma_*:G_{\rho_E} \To{\sim}G_{\rho_{^{\sigma}\!\!E}}$ of groups, which follows at once from the functoriality of $\sigma_*$. Due to \cite{dix}, Theorem 10.5, every continuous morphism of $p$-adic groups is analytic. By Corollary \ref{galo} one gets the second assertion.
\end{proof}

Let us denote by $\frakg_{\rho_E}$ resp. $\mathfrak{g}_{\rho_{^{\sigma}\!\!E}}$ the Lie algebras of $G_{\rho_E}$ resp. $G_{\rho_{^{\sigma}\!\!E}}$. The morphism $\sigma_*: G_{\rho_E} \to G_{\rho_{^{\sigma}\!\!E}}$ defines a morphism of Lie algebras $\sigma_*: \frakg_{\rho_E} \to \mathfrak{g}_{\rho_{^{\sigma}\!\!E}}$. After a change of basis of the $\bbC_p$-vector space $E_x$ we may assume that $G_{\rho_E}$ lies in $\mbox{GL}_r(\frako)$. Since the $p$-adic logarithm $\mbox{log}_p: \mbox{GL}_r(\frako) \to M_r(\bbC_p)$ commutes with the Galois actions, Lemma \ref{lie} implies that $\frakg_{\rho_E}$ and $\mathfrak{g}_{\rho_{^{\sigma}\!\!E}}$ lie in $M_r(\bbC_p)$ and that $\sigma_*: \frakg_{\rho_E} \to \mathfrak{g}_{\rho_{^{\sigma}\!\!E}}$ is the morphism $A \mapsto \sigma(A)$ for a matrix $A \in \frakg_{\rho_E} \subseteq M_r(\bbC_p)$. Since the morphism $\mbox{Gal}_{\bbQ_p} \times M_r(\bbC_p) \to M_r(\bbC_p), \; (\sigma, A) \mapsto \sigma(A)$ is continuous, one obtains a continuous map $\mbox{Gal}_{\bbQ_p} \to \mbox{Iso}(\frakg_{\rho_E}, \mathfrak{g}_{\rho_{^{\sigma}\!\!E}})$.\\  

\begin{cor} \label{galois2}
There is a continuous map $$\mbox{Gal}_{\bbQ_p} \longto \mbox{Iso}(\frakg_{\rho_E}, \frakg_{\rho_{^{\sigma}\!\!E}}).$$ If $X$, $x \in X$ and $E$ are already defined over $K \supseteq \bbQ_p$, one obtains a continuous representation $$\mbox{Gal}_K \longto \mbox{Aut}(\frakg_{\rho_E}).$$
\end{cor}

Let us from now on assume that all conditions of Proposition \ref{aaa} are fulfilled and that further $X$, $x \in X$ and $E$ are defined over some field extension $K \supseteq \bbQ_p$. 
The $\mbox{Gal}_K$-equivariant embedding $\frakg \into \mbox{End}_{\bbC_p}(V)$ implies that the upper exact row in diagram (\ref{dia2}) consists of $\mbox{Gal}_K$-modules. In fact, the Galois group $\mbox{Gal}_K$ acts on the radical $r(\frakg)$ because $\sigma \in \mbox{Gal}_K$ is an automorphism of the Lie algebra $\frakg$ and it maps $r(\frakg)$ again onto some maximal solvable ideal. But since the radical is the unique maximal solvable ideal of $\frakg$, see \cite{bour3}, Chapitre I, page 63, we find that $\sigma(r(\frakg)) = r(\frakg)$. The lower exact row of diagram (\ref{dia2}) then consists of semilinear $\mbox{Gal}_K$-representations, where a finite-dimensional $\bbC_p$-vector space $V$ is called a semilinear $\mbox{Gal}_K$-representation if there is a continuous morphism $\mbox{Gal}_K \times V \to V$ such that for all $\sigma \in \mbox{Gal}_K$ we have $\sigma(\lambda v) = \sigma(\lambda)\sigma(v)$ for all $\lambda \in \bbC_p, v \in V$, and $\sigma(v_1 + v_2) = \sigma(v_1) + \sigma(v_2)$ for all $v_1,v_2 \in V$. Furthermore, all vertical embeddings are $\mbox{Gal}_K$-equivariant. With the same reasoning one sees that diagram (\ref{dia1}) is a commutative diagram of semilinear $\mbox{Gal}_K$-modules and $\mbox{Gal}_K$-equivariant morphisms.\\

A semilinear $\mbox{Gal}_K$-representation $V$ is of Hodge-Tate type  if there exists a basis $v_1,\dots,v_n$ of $V$ with $\sigma(v_k) = \chi(\sigma)^{i_k}v_k$ for $i_k \in \bbZ$ and all $\sigma \in \mbox{Gal}_K$, where $\chi: G_K \to \bbZ_p^*$ denotes the cyclotomic character, see e.g. \cite{ser1}. The integers $i_k$ are called the Hodge-Tate weights of this representation.

We will need the following Lemma.

\begin{lemma} \label{gitter}
Let $\frakg \into \mbox{End}_{\bbC_p}(V)$ be a semisimple Lie algebra over $\bbQ_p$ and let $L \supseteq \bbQ_p$ be a field extension  such that $\frakg \tensor L$ is split semisimple. Then $\frakg(L)$ is a complete $L$-lattice in $\frakg({\bbC_p})$, i.e. $\frakg(L) \tensor_L \bbC_p = \frakg({\bbC_p})$.
\end{lemma}

\begin{proof}
Consider the surjective morphism $\frakg \tensor L \onto \frakg(L)$. Hence $\frakg(L) \subset \mbox{End}_{\bbC_p}(V)$  is a split semisimple Lie algebra over $L$. Let $\frakg(L) = \frakg_1 \oplus \dots \oplus \frakg_n$ with split simple Lie algebras $\frakg_i$. These are totally simple due to \cite{bour3}, Chapitre VIII, \S3.2, Corollaire 2, i.e. $\frakg_i \tensor_L \bbC_p$ is also simple. Let us assume that the map $\frakg(L) \tensor_L \bbC_p \onto \frakg_{\bbC_p}$ has a nontrivial kernel. Without restriction of generality we have $\frakg_1 \tensor \bbC_p \cap \frakg_2 \tensor \bbC_p \neq \{0\}$ in $\mbox{End}_{\bbC_p}(V)$, and it follows that $\frakg_1 \tensor \bbC_p \subseteq \frakg_2 \tensor \bbC_p$ since both are simple Lie algebras. Furthermore, we have $[\frakg_1, \frakg_2] = \{0\}$, hence $[\frakg_1 \tensor \bbC_p, \frakg_2 \tensor \bbC_p] = \{0\}$. Then $$\frakg_1 = [\frakg_1, \frakg_1] \subseteq [\frakg_1 \tensor \bbC_p, \frakg_1 \tensor \bbC_p] \subseteq [\frakg_1 \tensor \bbC_p, \frakg_2 \tensor \bbC_p] = \{0\},$$ which is a contradiction, so the kernel of this map has to be trivial. 
\end{proof}  

\begin{prop} \label{qwer}
For the $\mbox{Gal}_K$-modules $\frakg_{\rho_E}$ in Corollary \ref{galois2} the following holds:
\begin{enumerate} 
\item If $\frakg_{\rho_E}$ is abelian, then there exists a finite \'etale covering $Y \to X$ and a surjective morphism $T_pJ\tensor_{\bbZ_p}\bbQ_p \longonto \frakg_{\rho_E}$ of Galois-modules, where $J$ denotes the Jacobian of the curve $Y$. In particular, $\frakg_{\rho_E} \tensor_{\bbQ_p} \bbC_p$ is of Hodge-Tate type with the only possible Hodge-Tate weights $0$ und $1$.   
\item If $\frakg_{\rho_E}$ is a split semisimple Lie algebra with splitting field $L \supseteq \bbQ_p$, and if $M$ is the smallest field containing $K$ and $L$, then the semilinear $\mbox{Gal}_M$-representation $\frakg_{{\rho_E},L} \tensor_L \bbC_p$ is of Hodge-Tate type with weight $0$.\end{enumerate}
\end{prop}    

\begin{proof}
\begin{enumerate} \item Let $H \subseteq G_{\rho_E}$ be an open, uniform pro-$p$ subgroup. Then there is a finite \'etale covering $Y \to X$ with a $\mbox{Gal}_K$-equivariant and surjective morphism $\pi_1(Y,y) \onto H$. Since $H$ is abelian, this morphism factors through $\pi_1^{ab}(Y,y) \onto H$. If $J$ denotes the Jacobian variety of the curve $Y$ and $TJ$ its Tate module, then there is a $\mbox{Gal}_K$-equivariant isomorphism $TJ \iso \pi_1^{ab}(Y,y)$ (see \cite{miln} Proposition 9.1), and since $H$ is a pro-$p$ group, one gets a surjection $T_pJ \onto H.$ If we consider the free $\bbZ_p$-module $T_pJ$ as an abelian pro-$p$ group, we get a $\mbox{Gal}_K$-equivariant surjective morphism between the Lie algebras $\frakg(T_pJ) = T_pJ \tensor \bbQ_p \onto \frakg(H) = \frakg_{\rho_E}.$ Finally, by \cite{ser2} we know that $T_pJ$ is of Hodge-Tate type with weights $0$ and $1$, so its quotient $\frakg_{\rho_E}$ also possesses at most the weights $0$ and $1$.
    
\item  By Lemma \ref{gitter}, after tensoring with $\bbC_p$ the $\mbox{Gal}_M$-module $\frakg_{\rho_E,L}$ equals the semilinear $\mbox{Gal}_K$-module $\frakg_{\rho_E, \bbC_p} \subseteq \mbox{End}_{\bbC_p}(E_x)$. Now $\mbox{End}_{\bbC_p}(E_x) \iso M_r(\bbC_p)$ is the trivial semilinear module, hence of Hodge-Tate type with weights 0, and the claim for the semilinear submodule $\frakg_{\rho_E, \bbC_p}$ follows at once. \end{enumerate} 
\end{proof}

We will  give a description of the Galois actions occuring in Proposition \ref{qwer}. We denote by $\overline{\psi(I_K)}$ the Zariski closure of the image of the inertia group $I_K$ of $\mbox{Gal}_K$ with respect to the representation $\psi:G_K \to \mbox{Aut}(\frakg_{\rho_E})$. J. P.  Serre proved in \cite{ser1} that the $p$-adic analytic group $\psi(I_K)$ is open in $\overline{\psi(I_K)}$, i.e. its Lie algebra is algebraic. Furthermore, for semilinear modules as in Proposition \ref{qwer}$.1$, if the representation $\mbox{Gal}_K \to \mbox{Aut}(\frakg_{\rho_E})$ is semisimple, the irreducible root systems of $\overline{\psi(I_K)}^0$ can only be the classical ones of type $A,B,C$ or $D$. For an irreducible $\mbox{Gal}_K$-module of odd dimension the root system even has to be of type $A$ (see, \cite{ser1}, Th\'eor\`eme 7). If the semilinear $\mbox{Gal}_K$-module is of Hodge-Tate type with weights $0$, we can describe the $\mbox{Gal}_K$ action using the work of S. Sen. Assume that the semisimple Lie algebra in Proposition \ref{qwer}.$2$ splits already over $\bbQ_p$. Then $\psi(I_K)$ is finite, hence there is a finite field extension $L \supseteq K$  such that the inertia group of $G_L$ acts trivially on $G_{\rho_E}$, see \cite{sen1}, Theorem 11.

\bibliography{literatur}
\bibliographystyle{alpha}

\end{document}